\documentclass[11pt]{amsart}
 
\usepackage{amssymb, amscd, txfonts}
\usepackage{graphicx}

\numberwithin{equation}{section}

\newtheorem{theorem}{Theorem}[section]
\newtheorem{proposition}[theorem]{Proposition}
\newtheorem{lemma}[theorem]{Lemma}
\newtheorem{corollary}[theorem]{Corollary}

\theoremstyle{definition}
\newtheorem{definition}[theorem]{Definition}
\newtheorem{example}[theorem]{Example}

\theoremstyle{remark}
\newtheorem{remark}[theorem]{Remark}

\renewcommand{\hom}{\operatorname{Hom}}

\newcommand{\Z}{\mathbb{Z}}
\newcommand{\Q}{\mathbb{Q}}
\newcommand{\R}{\mathbb{R}}
\newcommand{\C}{\mathbb{C}}
\newcommand{\F}{\mathbb{F}}

\newcommand{\proj}{{\mathbb P}}

\newcommand{\moduli}{\mathcal{M}_{r,a}}
\newcommand{\cover}{\widetilde{\mathcal{M}}_{r,a}}
\newcommand{\cove}{\widetilde{\mathcal{M}}}
\newcommand{\Mcub}{\mathcal{M}_{{\rm cub}}}
\newcommand{\Mmcub}{\widetilde{\mathcal{M}}_{{\rm cub}}}
\newcommand{\lift}{\widetilde{\mathcal{P}}}
\newcommand{\SL}{{\rm SL}}

\newcommand{\PGL}{{\rm PGL}}

\newcommand{\Oline}{\mathcal{O}_{{\mathbb P}^{1}}}
\newcommand{\Oplane}{\mathcal{O}_{{\mathbb P}^{2}}}
\newcommand{\Ospace}{\mathcal{O}_{{\mathbb P}^{3}}}

\newcommand{\OHirze}{\mathcal{O}_{\mathbb{F}_{n}}}
\newcommand{\sheaf}{\mathcal{O}}

\newcommand{\Pic}{{\rm Pic}}
\newcommand{\Or}{{\rm O}}

\DeclareMathOperator{\aut}{Aut}

\begin{document}

%%%%%%% Title %%%%%%%%%%%%%%%%%%%%%%%%
\title[]{Rationality of the moduli spaces of Eisenstein $K3$ surfaces}
\author[]{Shouhei Ma}
\thanks{S. M. was supported by Grant-in-Aid for JSPS fellows [21-978] and Grant-in-Aid for Scientific Research (S), No 22224001.}  
\address{Graduate School of Mathematics, Nagoya University, Nagoya 464-8602, Japan}
\email{ma@math.nagoya-u.ac.jp}
\author[]{Hisanori Ohashi}
\thanks{H. O. was supported by Grant-in-Aid for Scientific Research (S), No 22224001 and for Young Scientists (B) 23740010.}  
\address{Department of Mathematics, Faculty of Science and Technology, Tokyo University of Science, 2641 Yamazaki, Noda, Chiba, 278-8510, Japan}
\email{ohashi@ma.noda.tus.ac.jp}
\author[]{Shingo Taki}
\address{School of Information Environment, Tokyo Denki University, 2-1200 Muzai Gakuendai, Inzai-shi, Chiba 270-1382, Japan}
\email{staki@mail.dendai.ac.jp}
%\dedicatory{}
\maketitle 

\begin{abstract}
$K3$ surfaces with non-symplectic symmetry of order $3$ are classified by 
open sets of twenty-four complex ball quotients associated to Eisenstein lattices. 
We show that twenty-two of those moduli spaces are rational. 
\end{abstract}

\maketitle

%%%%%%%%%%%%%%
% Section : Introduction 
%%%%%%%%%%%%%%

\section{Introduction}

The study of $K3$ surfaces with non-symplectic symmetry has arisen as an application of the Torelli theorem, 
and by now, it has been recognized as closely related to classical geometry and special arithmetic quotients. 
They were first studied systematically in the involution case by Nikulin \cite{Ni}, 
who classified their topological types using the lattices of $2$-cycles (anti-)invariant under the involutions. 
The anti-invariant lattices also provide the period domains, which are Hermitian symmetric of type IV and 
whose quotients by the orthogonal groups give the moduli spaces.

What comes next to the involution case is 
the case of automorphisms of order $3$. 
Kond\=o, Dolgachev, and van Geemen \cite{D-G-K}, \cite{Ko} studied two moduli spaces of 
$K3$ surfaces with such symmetry, in connection with genus $4$ curves and cubic surfaces. 
Subsequently, Artebani-Sarti \cite{A-S} and the third-named author \cite{Ta} gave a topological classification of such automorphisms. 
Let $X$ be a $K3$ surface with a non-symplectic symmetry $G\subset{\aut}(X)$, $G\simeq{\Z}/3{\Z}$. 
Let $L(X, G)\subset H^2(X, {\Z})$ be the lattice of $G$-invariant cycles, 
and $E(X, G)\subset H^2(X, {\Z})$ be its orthogonal complement. 
These lattices have 3-elementary discriminant groups, analogous to the 2-elementary property in the involution case. 
What is more crucial is that $E(X, G)$ is endowed with the structure of an Eisenstein lattice, namely a Hermitian form over the ring of Eisenstein integers. 
Then a result of \cite{A-S} and \cite{Ta} says that 
the topological types of such pairs $(X, G)$ are, by associating $E(X, G)$, 
in one-to-one correspondence with certain Eisenstein lattices embeddable in the $K3$ lattice. 
In view of this, we shall call such a pair $(X, G)$ an \textit{Eisenstein $K3$ surface}. 
According to \cite{A-S}, \cite{Ta}, $E(X, G)$ is in turn encoded in the pair $(r, a)$ 
where $r$ is the rank of $L(X, G)$ and $a$ is the length of its discriminant group, and there are exactly twenty-four such $(r, a)$.

For each $(r, a)$, the period domain for Eisenstein $K3$ surfaces $(X, G)$ of that type 
is the complex ball associated to $E(X, G)$. 
One obtains the moduli space ${\moduli}$ of those Eisenstein $K3$ surfaces as the quotient of the ball by the unitary group of $E(X, G)$, 
with a Heegner divisor removed. 
This story is similar to the involution case, but note that the types of period domains are different.

In this article we study the birational types of ${\moduli}$. 
The spaces $\mathcal{M}_{2,2}$ and $\mathcal{M}_{12,5}$, studied in \cite{Ko} and \cite{A-C-T}, \cite{D-G-K} respectively, 
have been known to be rational by the corresponding results for the moduli of genus $4$ curves 
(\cite{SB}) and of cubic surfaces (classical). 
We show that this property actually holds for most ${\moduli}$. 

\begin{theorem}\label{main}
The moduli space ${\moduli}$ of Eisenstein $K3$ surfaces of type $(r, a)$ is rational, possibly except for $(r, a)=(8, 7)$ and $(10, 6)$. 
\end{theorem}

A similar rationality result is known in the involution case (\cite{Ko0}, \cite{Ma2}, \cite{D-K}).  
It is natural to expect analogous results for other non-symplectic symmetry, 
and the present article goes into the Eisenstein case. 
In fact, it appears that automorphisms of order $2$ and $3$ cover a wide range of non-symplectic automorphisms: 
as the order grows, there seem to be only a few number of moduli spaces, of rather small dimension 
(though the classification is not yet completed).  
%They in general get rarer as the order grows, though complete classification is not yet obtained. 

We will prove Theorem \ref{main} case-by-case. 
A basic strategy is to first find a \textit{canonical} triple cover construction of general members of ${\moduli}$ 
using $-\frac{3}{2}K_Y$-curves on a Hirzebruch surface $Y={\F}_{2N}$, $0\leq N\leq3$. 
More precisely, we consider an explicit locus $U\subset|-\frac{3}{2}K_Y|$ parametrizing 
curves with prescribed type of singularities and irreducible decomposition. 
We obtain a period map 
\begin{equation*}
\mathcal{P} : U/{\aut}({\F}_{2N}) \dashrightarrow {\moduli} 
\end{equation*}
by taking the resolutions of cyclic triple covers of ${\F}_{2N}$ branched over $B\in U$. 
We can calculate the degree of such maps $\mathcal{P}$ in a systematic manner (see \S \ref{ssec: recipe}). 
If we could find those $U$ with $\deg(\mathcal{P})=1$, 
the problem is reduced to the rationality of $U/{\aut}({\F}_{2N})$ which we prove by studying the ${\aut}({\F}_{2N})$-action.

This strategy is analogous to the one in the involution case \cite{Ma2}, 
but hidden behind the similarity are some subtle features in the present case. 
The first is the existence of isolated fixed points of $(X, G)=\mathcal{P}(B)$, 
which appear over the singular points of $B$. 
By the above construction, we keep away from such fixed points, in a sense. 
Secondly, asking the triple cover to have canonical singularities is a strong demand, 
so that the singularities of $B$ are quite limited (at worst ramphoid cusps). 
Finally, smooth rational surfaces $Y$ with $3K_Y\in2{\rm Pic}(Y)$ are rare: 
they are only ${\F}_{2N}$.

The above easy construction offers period maps of degree $1$ for as many as seventeen ${\moduli}$, 
but does not cover all cases. 
To analyze the remaining five 
($\mathcal{M}_{4,3}$, $\mathcal{M}_{6,4}$,  $\mathcal{M}_{8,5}$,  $\mathcal{M}_{10,4}$ and $\mathcal{M}_{12,3}$), 
we develop a theory of branch \textit{curves} that deals with isolated fixed points more substantially. 
This is the notion of \textit{mixed branch}. 
It contains and is more flexible than $-\frac{3}{2}K_Y$-curves, 
and using it we can work with fixed curves and isolated fixed points quite satisfactorily. 
Those five ${\moduli}$ are provided with birational period maps using mixed branch.

The rationality problem is open for $\mathcal{M}_{8,7}$ and $\mathcal{M}_{10,6}$. 
They are unirational by the constructions in \cite{A-S}, \cite{A-S-T}. 
Unfortunately, for those two we failed to find a canonical and effective construction as above, 
due to which we could not approach them.

The rest of the article is as follows. 
\S \ref{sec: preliminary} contains the preliminaries on Eisenstein lattices and automorphisms of ${\F}_{n}$. 
In \S \ref{sec: EK3} we recall/reformulate basic results on Eisenstein $K3$ surfaces. 
We introduce mixed branches in \S \ref{ssec:mixed branch}, 
and then study $-\frac{3}{2}K_{{\F}_{2N}}$-curves in \S \ref{ssec:pure branch}. 
The method of degree calculation is explained in \S \ref{ssec: recipe}. 
After these preliminaries, the proof of Theorem \ref{main} begins in \S \ref{sec:g=5}. 
We proceed according to the maximal genus $g$ of fixed curves: the cases with genus $g$ are treated in \S $10-g$. 
We adopt this division policy because it exhibits the degeneration relations among the moduli spaces with a common $g$.

Throughout this article we shall denote by $A_n$, $D_m$, $E_l$ 
the \textit{negative}-definite root lattice of type $A_n$, $D_m$, $E_l$ respectively. 
We denote by $U$ the even indefinite unimodular lattice of rank $2$. 

\vspace{0.3cm}

\noindent 
\textbf{Acknowledgement.}
H.~O. is grateful to Professor Kond\=o for his encouragements.

%%%%%%%%%%%%%%
% Section : Preliminaries 
%%%%%%%%%%%%%%

\section{Preliminaries}\label{sec: preliminary}

In this section we prepare some results on Eisenstein lattices (\S \ref{ssec: E lattice}) 
and automorphisms of Hirzebruch surfaces (\S \ref{ssec: Hirze}). 
They are a technical basis for the rest of the article. 
The reader may skip for the moment and return when necessary.

%%%%%%%%%%%%%%%%
% Subsection: Eisenstein lattice
%%%%%%%%%%%%%%%%

\subsection{Eisenstein lattices}\label{ssec: E lattice}

Let $E$ be an {\em{even lattice}}, namely a free $\mathbb{Z}$-module endowed 
with a nondegenerate integral symmetric bilinear form $(\ ,\ )$ such that
$(l,l)\in 2\mathbb{Z}$ for every $l\in E$. 
A structure of {\em{Eisenstein lattice}} on $E$ is a self-isometry $\rho$ of $E$ of order $3$, 
such that $\rho (l) \neq l$ for any $0\neq l\in E$. 
Equivalently, a self-isometry $\rho$ gives an Eisenstein structure if 
it satisfies $\rho^2+\rho+\mathrm{id}=0$.
In this subsection we study some properties of such a pair $(E, \rho)$.

First we justify the naming ''Eisenstein lattice''. 
Let $R={\Z}[\zeta]$, $\zeta=e^{2\pi i/3}$, be the ring of Eisenstein integers.  
For an Eisenstein lattice $(E,\rho )$ as above, 
the ${\Z}$-module $E$ is naturally equipped with an $R$-module structure by $\zeta \cdot l = \rho (l)$. 
Then we have an $R$-valued Hermitian form on $E$ by 
\begin{equation}\label{eqn:Eisen lattice I}
(l,l')_{\mathcal{E}} := (l,l')+\zeta (l,\rho (l')) + \zeta^2 (l,\rho^2 (l')) \in R. 
\end{equation}
%It is clear that $(\zeta l, l')_{\mathcal{E}}=\zeta(l, l')_{\mathcal{E}}$. 
If we decompose $E\otimes_{\Z}{\C}=V\oplus\overline{V}$ by the $\rho$-action where $\rho|_V=\zeta$, 
and consider the projection $\pi: E\to V$, 
then we have $(l,l')_{\mathcal{E}}=3(\pi(l), \overline{\pi(l')})$. 

Conversely, if $E$ is a free $R$-module equipped with a Hermitian form $(\ ,\ )_{\mathcal{E}}$, 
the symmetric bilinear form 
\begin{equation}\label{eqn:Eisen lattice II} 
(l,l') := \frac{2}{3} \Re e ((l,l')_{\mathcal{E}})
\end{equation}
defines (in general not integral) a lattice structure 
on the $\mathbb{Z}$-module $E$ which naturally has an Eisenstein structure $\rho$ defined by the action of $\zeta$. 
One checks that the constructions \eqref{eqn:Eisen lattice I} and \eqref{eqn:Eisen lattice II} 
are converse to each other. 
The bilinear form $(\ ,\ )$ is even if and only if the Hermitian form $(\ ,\ )_{\mathcal{E}}$ satisfies 
\begin{equation}\label{eqn:Eisen lattice III}
(l, l)_{\mathcal{E}}\in3{\Z}
\end{equation} 
for all $l\in E$. 
Thus Eisenstein lattices in our sense naturally correspond to Hermitian lattices over $R$ 
with the property \eqref{eqn:Eisen lattice III}. 
In this article we will work rather in the category of quadratic forms. 
Note that the signature of $E$ as a quadratic form is twice that of $E$ as a Hermitian form. 

We denote by $E^{\vee}={\hom}_{\Z}(E, {\Z})$ the dual quadratic form (inside $E\otimes_{\Z}{\Q}$), 
and $E^{\ast}={\hom}_R(E, R)$ the dual Hermitian form (inside $E\otimes_R{\Q}(\zeta)$). 
If we identify $E\otimes_{\Z}{\Q}$ with $E\otimes_R{\Q}(\zeta)$ naturally, 
then $E^{\vee}$ is equal to $\sqrt{-3}E^{\ast}$. 
Let $A_E=E^{\vee}/E$ be the discriminant group of $E$, 
which is endowed with the discriminant form $q_A: A_E\to{\Q}/2{\Z}$.

\begin{example}\label{eee}
(1) A fundamental example is the root lattice $E=A_2$. Up to taking square, it has a
unique isometry $\rho$ of order $3$ which gives $E$ the structure of an Eisenstein lattice.
The corresponding Hermitian form of rank $1$ is $\langle -3\rangle$. \\
%In what follows, we always consider $A_2$ as being equipped with this structure.\\
%it is generated by $\mathsf{e}_1, \mathsf{e}_2$ with the bilinear relations $(\mathsf{e}_i,
%\mathsf{e}_i)=-2$ $(i=1,2)$ and $(\mathsf{e}_1,\mathsf{e}_2)=1$. 
%We easily see that the transformation $\rho =
%\begin{pmatrix}
%0 & -1 \\ 1 & -1
%\end{pmatrix}$
(2) %By the uniqueness theorems for indefinite lattices \cite{Ni0}, 
Since we have an isometry 
$A_2\oplus A_2(-1)\simeq U \oplus U(3)$ 
of quadratic forms, 
by (1) the lattice $E=U\oplus U(3)$ has the structure of an Eisenstein lattice 
which corresponds to the Hermitian form $\langle3, -3\rangle$. 
Moreover, since $\rho$ acts trivially on the 
discriminant group $A_E$, it preserves the overlattices of $E$ which are isomorphic to $U\oplus U$. 
Hence we also obtain an Eisenstein structure on $U\oplus U$. \\
(3) Since the root lattices $E_6$ and $E_8$ both can be obtained as overlattices of some direct 
sum of $A_2$, by the same reasoning as (2), 
these have the structure of an Eisenstein lattice, too. 

We shall fix the above Eisenstein structures on $U\oplus U$, $U\oplus U(3)$, $E_6$ and $E_8$. 
\end{example}

The unitary group $\mathrm{U} (E)$ of an Eisenstein lattice $(E, \rho )$ is 
naturally embedded in the orthogonal group ${\rm O}(E)$ by 
\begin{equation*}
\mathrm{U}(E)= \{\gamma \in \mathrm{O} (E) \mid \gamma \circ \rho = \rho \circ \gamma\}.
\end{equation*}
%As $\rho$ determines the complex structure of $E$, 
%this ${\rm U}(E)$ is nothing but the unitary group of $E$ considered as a Hermitian lattice over ${\Z}[\zeta]$. 
In particular, we have a natural homomorphism ${\rm U}(E)\to{\rm O}(A_E)$ 
to the orthogonal group of the discriminant form. 
We prove that it is surjective for some special Eisenstein lattices, 
as an analogue of the surjectivity property of \cite{Ni0} for orthogonal groups.

\begin{proposition}\label{surj I}
(1) Let $E$ be the indefinite Eisenstein lattice $A_2(-1)\oplus A_2^n$. 
Then the homomorphism $\mathrm{U} (E)\rightarrow {\rm O}(A_E)$ is surjective.\\
(2) Let $E$ be the definite Eisenstein lattice $A_2^n$ with $n\leq3$. 
Then the homomorphism $\mathrm{U} (E)\rightarrow {\rm O}(A_E)$ is surjective. 
\end{proposition}

\begin{proof}
The groups ${\rm O}(A_E)$ are in fact full orthogonal groups in characteristic $3$. 
Our proof relies on the fundamental fact that 
they are generated by reflections in non-isotropic vectors (see, e.g., \cite{Ki} Chapter 1.2). 

Let $L$ be the odd unimodular lattice $\langle1\rangle^n\oplus\langle-1\rangle$ (resp. $\langle1\rangle^n$) in the case (1) (resp. (2)). 
Then $E$ can be identified with the tensor product $L\otimes A_2$, including the correspondence of Gram matrices. 
Using this tensor notation, the Eisenstein structure of $E$ has the form $\mathrm{id}_L \otimes \rho$, 
where $\rho$ is from Example \ref{eee} (1). 
Now for $g\in \mathrm{O} (L)$, we can define an element of ${\rm U}(E)$ by $\alpha (g) = g\otimes \mathrm{id}_{A_2}$. 
This defines an injective homomorphism $\alpha \colon \mathrm{O} (L)\rightarrow \mathrm{U}(E)$. 
Consider the composite of $\alpha$ and $\mathrm{U} (E)\rightarrow {\rm O}(A_E)$. 
By taking a natural basis of $A_E=A_{L\otimes A_2}$, 
it is identified with the reduction map $\beta\colon {\rm O}(L)\to{\rm O}(L/3L)$, 
where $L/3L$ is naturally equipped with a quadratic form over ${\Z}/3{\Z}$.

To prove the proposition, now
it suffices to show that the reduction map $\beta$ is surjective.
Let $(\ ,\ )$ be the bilinear form on $L$. Then the bilinear form on $L/3L$ is just given by 
$(\ ,\ ) \mod 3$, hence we use the same notation $(\ ,\ )$ for them. 

Since ${\rm O}(L/3L)$ is an orthogonal group in odd characteristic, it is generated by 
reflections $r_a$ for non-isotropic elements $a\in L/3L$, where
\begin{equation}\label{reflection}
r_a \colon x\mapsto x-\frac{2(x,a)}{(a, a)}a. 
\end{equation}
If $l\in L$ satisfies $(l, l)\in \{\pm 1, \pm 2\}$, 
then the reflection $r_l$ defined by the same formula as \eqref{reflection} gives an element of ${\rm O}(L)$, 
and its image in ${\rm O}(L/3L)$ is the reflection in $[l]\in L/3L$. 
Thus our surjectivity assertion is reduced to the "liftability of reflection vectors", 
that is, the following problem: for any non-isotropic element $a \in L/3L$, 
find a lift $l\in L$ of $a$ (or $2a$, since they define the same reflections) 
such that $(l, l)\in \{\pm 1, \pm 2\}$.
This purely arithmetic step is realized in the next lemma.
\end{proof}

\begin{lemma}
Let $L$ be the odd unimodular lattice of signature $(n,1)$ for (1), or of signature $(n,0)$ for
(2) respectively. In case (2) suppose $n\leq 3$. 
Then for any non-isotropic element $a\in L/3L$, there exists a lift $l\in L$ of $a$ or $2a$ such that 
$(l, l)\in \{\pm 1, \pm 2\}$. 
\end{lemma}

\begin{proof}
Case (2) is easily done by hand, so we prove only (1). We take the coordinates for $L$ so that 
the quadratic form on $L$ is given by 
\[q (x_0, \cdots, x_n)=-x_0^2+x_1^2+\cdots +x_n^2.\]
Let $(y_0,\cdots, y_n)\in L/3L$ be a given non-isotropic element. 
We have to show the existence of $l=(x_0, \cdots, x_n)\in L$ such that 
$(l, l)\in \{\pm 1, \pm 2\}$ and $x_i \mod 3$ is equal to the given $y_i$. 
This is purely an arithmetic problem. One solution is given as follows.

First we ignore the zero coordinates $y_i\equiv 0 (i>0)$ by using $x_i=0$.
Moreover for $y_i\equiv 1$ or $\equiv 2$, 
we can use $x_i=1, -2$ or $=-1, 2$ respectively so that 
$x_i^2$ takes the value $1$ or $4$ at any rate.
These two steps reduce the equation to 
\[ -x_0^2+(1+1+\cdots +4+4+\cdots)\in \{\pm 1, \pm 2\} \text{(exactly $n$ terms in the parentheses)}.\] 
 
When $y_0\equiv 0$, take the positive integer $s$ such that 
\[ 3s^2-6s+4 \leq [n/3] < 3(s+1)^2-6(s+1)+4.\]
(If $[n/3]=0$ then we take $s=0$.) 
Then putting $x_0=3s$ gives one solution to the above equation
\[-(3s)^2+1\cdot ([n/3]+n-3s^2)+4\cdot (3s^2-[n/3])=1\text{ or }2.\] 
(We can see that $[n/3]+n-3s^2\geq 4(3(s-1)^2+1)-3s^2=(3s-4)^2\geq 0$, and so on.)
When $y_0\equiv 1$, take the positive integer $s$ such that 
\[ 3s^2-4s+2 \leq [(n-1)/3] < 3(s+1)^2-4(s+1)+2.\]
(If $[(n-1)/3]=0$ then we take $s=0$.)
Then putting $x_0=3s+1$ gives one solution to the above equation
\[-(3s+1)^2+1\cdot ([(n-1)/3]+n-3s^2-2s)+4\cdot (3s^2+2s-[(n-1)/3])=1 \text{ or } 2.\]
Finally when $y_0\equiv -1$, we can find $x$ with $x\equiv -y$ by previous argument. 
All the cases are covered and the lemma is proved.
\end{proof}

As a consequence of Proposition \ref{surj I}, we have the following. 

\begin{corollary}\label{surj II}
Let $E$ be one of the following Eisenstein lattices: 
\begin{equation*}
A_2(-1)\oplus A_2^n\oplus E_8^m, \quad U^2\oplus A_2^l\oplus E_8^k \; (l\leq3). 
\end{equation*}
Then the natural homomorphism ${\rm U}(E)\to {\rm O}(A_E)$ is surjective. 
\end{corollary}

Among the twenty-four Eisenstein lattices associated to Eisenstein $K3$ surfaces, 
twenty-two excepting $U^2\oplus A_2^4$ and $U^2\oplus A_2^5$ may be written in the above form (see \S \ref{ssec: EK3}). 
We will see in \S \ref{sec:g=1} that the surjectivity property also holds for those two, 
by geometric arguments.

%%%%%%%%%%%%%
% Subsection: Hirzebruch
%%%%%%%%%%%%%

\subsection{Automorphisms of Hirzebruch surfaces}\label{ssec: Hirze}

We recall some basic facts about Hirzebruch surfaces 
(see, e.g., \cite{Ma2} \S 3 for more detail). 
For $n\geq0$ let ${\F}_n={\proj}({\Oline}(n)\oplus {\Oline})$ be the $n$-th Hirzebruch surface 
with the natural projection $\pi\colon{\F}_n \to {\proj}^1$. 
The ${\proj}^1$-fibration $\pi$ has a $(-n)$-section $\Sigma$ (which is unique in case $n>0$), 
and also a section $H_0$ with $(H_0, H_0)=n$ that is disjoint from $\Sigma$. 
The Picard group of ${\F}_n$ is freely generated by $H_0$ and a fiber $F$ of $\pi$. 
We shall denote $L_{a,b}={\OHirze}(aH_0+bF)$. 
%We have $(L_{a,b}, F)=a$ and $(L_{a,b}, \Sigma)=b$. 
For example, $\Sigma$ belongs to $|L_{1,-n}|$; 
the canonical bundle $K_{{\F}_n}$ is isomorphic to $L_{-2,n-2}$.  

We take two distinct $\pi$-fibers $F_0$, $F_{\infty}$ and set 
\begin{eqnarray*}
U_1={\F}_n\backslash (F_{\infty}+H_0), & &   U_2={\F}_n\backslash (F_0+H_0), \\ 
U_3={\F}_n\backslash (F_{\infty}+\Sigma),    & &   U_4={\F}_n\backslash (F_0+\Sigma).    
\end{eqnarray*} 
These open sets are isomorphic to ${\C}^2$ and form a covering of ${\F}_n$. 
Each $U_i$ has a coordinate $(x_i, y_i)$ with the transformation rules  
\begin{equation*}
x_1=x_3=x_2^{-1}=x_4^{-1}, 
\end{equation*}
\begin{equation*}
y_3=y_1^{-1}, \quad y_4=y_2^{-1}, \quad y_2=x_1^ny_1, \quad y_4=x_3^{-n}y_3, 
\end{equation*}
and such that $\pi$ is given (inhomogeneously) by $(x_i, y_i)\mapsto x_i$. 

The restriction to $U_3$ of a curve $C\subset{\F}_n$ is defined by $F(x_3, y_3)=0$ for a polynomial $F$ of $x_3, y_3$. 
This identifies $H^0(L_{a,b})$ for $a, b\geq0$ with the following linear space of polynomials, up to constant: 
\begin{equation}\label{def eq}
\left\{ \sum_{i=0}^{a}f_i(x_3)y_3^{a-i}, \: {\rm deg}f_i\leq b+in \right\}.
\end{equation} 
If $C\in|L_{a,b}|$ is defined by $\sum_if_i(x_3)y_3^{a-i}=0$ on $U_3$,  
then on $U_1$  (resp. $U_4$, $U_2$) it is defined by 
$\sum_if_i(x_1)y_1^i=0$  
(resp. $\sum_if_i(x_4^{-1})x_4^{b+in}y_4^{a-i}=0$, 
$\sum_if_i(x_2^{-1})x_2^{b+in}y_2^i=0$).

\textit{In the rest of this section we assume $n>0$.} 
Then %the automorphism group ${\aut}({\F}_n)$ preserves $\Sigma$ so that 
we have the exact sequence 
\begin{equation}\label{Aut(Hir)}
1 \to R \to {\aut}({\F}_n) \to {\aut}(\Sigma) \to 1 
\end{equation} 
where $R={\aut}({\Oline}(n)\oplus {\Oline})/{\C}^{\times}$. 
This sequence splits when $n$ is even. 
The group $R$ is isomorphic to ${\C}^{\times}\ltimes H^0({\Oline}(n))$ and consists of the automorphisms  
%\begin{equation}\label{desrcibe R}
%R = \left\{ g_{\alpha,s}=\begin{pmatrix} \alpha &    s  \\
%                                             0    &    1  \end{pmatrix}, \; \; \alpha\in{\C}^{\times}, s\in {\Hom}({\Oline}, {\Oline}(n)) \right\} .
%\end{equation}    
%In particular, $R\simeq {\C}^{\times}\ltimes{\C}^{n+1}$.                                           
%If $s\in H^0({\Oline}(n))$ is expressed as $s=\sum_{i=0}^{n}\lambda_iX^iY^{n-i}$, 
%then in the coordinate $(x_3, y_3)$ of $U_3$ the action of $g_{\alpha, s}$ is expressed by 
\begin{equation}\label{$R$-action in coordinate}
g_{\alpha, s} : U_3\ni(x_3, y_3) \mapsto (x_3, \alpha y_3 + \sum_{i=0}^{n}\lambda_ix_3^i)\in U_3,  
\end{equation}
where $\alpha\in{\C}^{\times}$ and $s=\sum_{i=0}^{n}\lambda_ix^i\in H^0({\Oline}(n))$. 
Later we will also use the following automorphisms: 

\begin{equation}\label{auto coord 1}
h_{\beta} : U_3\ni(x_3, y_3) \mapsto (\beta x_3, y_3)\in U_3, \quad \beta\in{\C}^{\times}, 
\end{equation}
%\begin{equation}\label{auto coord 2}
%\sigma : U_3\ni(x_3, y_3) \mapsto (1-x_3, y_3)\in U_3, 
%\end{equation}
\begin{equation}\label{auto coord 3}
\iota : U_3\ni(x_3, y_3) \mapsto (x_3, y_3)\in U_4, 
\end{equation}
\begin{equation}\label{auto coord 4}
i_{\lambda} : U_2\ni(x_2, y_2) \mapsto (x_2+\lambda, y_2) \in U_2, \quad \lambda\in{\C}. 
\end{equation}
%
%In \eqref{auto coord 3} the second $(x_3, y_3)$ is the coordinate in $U_4$. 
These rational maps actually extend to automorphisms of ${\F}_n$. 
%Note that they all preserve $H_0$. 
%We can use $x_3=x_2^{-1}$ as an affine coordinate of $H_0\simeq\Sigma$ to see their action on $H_0\simeq\Sigma$. 

We will need to know the action of ${\aut}({\F}_n)$ on some spaces.

\begin{lemma}\label{stab of (pt, *)}
The group ${\aut}({\F}_n)$ acts on ${\F}_n$ (resp. ${\F}_n\times\Sigma$, ${\F}_n\times{\F}_n$) 
almost transitively with the stabilizer $G$ of a general point being connected and solvable. 
\end{lemma}

\begin{proof}
The almost transitivity is checked immediately. 
Let $p_i$ denote the point $(x_i, y_i)=(0, 0)$ in $U_i$. 
We may normalize a general point of ${\F}_n$ (resp. ${\F}_n\times\Sigma$, ${\F}_n\times{\F}_n$) 
to be $p_3$ (resp. $(p_3, p_2)$, $(p_3, p_4)$). 
In view of the exact sequence 
\begin{equation}\label{seq:stab of pt}
0\to G\cap R \to G \to {\rm Im}(G\to{\aut}(\Sigma))\to 1, 
\end{equation}
it suffices to show that both $G_1=G\cap R$ and $G_2={\rm Im}(G\to{\aut}(\Sigma))$ are connected and solvable. 
In the case of ${\F}_n$, 
$G_2$ is the stabilizer in ${\aut}(\Sigma)$ of $p_1$ and hence isomorphic to ${\C}^{\times}\ltimes{\C}$, 
while $G_1$ is $\{ g_{\alpha,s}\in R  \, | \,  \lambda_0=0\}$ which is isomorphic to ${\C}^{\times}\ltimes{\C}^n$. 
In the case of ${\F}_n\times\Sigma$, 
$G_2$ is the stabilizer of the two ordered points $(p_1, p_2)$ and thus isomorphic to ${\C}^{\times}$, 
while $G_1$ is the same as the case of ${\F}_n$. 
Finally, in the case of ${\F}_n\times{\F}_n$, 
$G_2$ is the same as the case of ${\F}_n\times\Sigma$, 
and $G_1$ is $\{ g_{\alpha,s}\in R \, | \, \lambda_0=\lambda_n=0\}$ which is isomorphic to ${\C}^{\times}\ltimes{\C}^{n-1}$. 
\end{proof}

%We denote by $|L_{a,b}|_{sm}\subset|L_{a,b}|$ the locus of smooth curves. 
%\textbf{Tentative version. Move to each proof/non-prop style/improve/make coarse if necessary.}

%\begin{lemma}[cf. \cite{Ma2}]\label{linearization}
%When $n$ is odd, every line bundle on ${\F}_n$ is ${\aut}({\F}_n)$-linearized. 
%When $n$ is even, the bundle $L_{a, b}$ admits an ${\aut}({\F}_n)$-linearization if $b$ is even. 
%\end{lemma} 

\begin{lemma}[cf. \cite{Ma2}]\label{linear system}
We have the following. 

$(1)$ ${\aut}({\F}_n)$ acts on $|L_{0,1}|\simeq\Sigma$ transitively with connected and solvable stabilizer. 

$(2)$ ${\aut}({\F}_n)$ acts transitively on the open locus in $|L_{1,0}|$ of smooth curves. 
If $G$ is the stabilizer of $H_0\in|L_{1,0}|$, the natural homomorphism $G\to{\aut}(\Sigma)$ is surjective with 
the kernel $\{ g_{\alpha,0} \, | \, \alpha\in{\C}^{\times}\}$.  

$(3)$ Let $U\subset|L_{2,0}|$ be the open locus of smooth curves. 
A geometric quotient $U/{\aut}({\F}_n)$ exists 
and is naturally isomorphic to the moduli space $\mathcal{H}_{n-1}$ of hyperelliptic curves of genus $n-1$. 
\end{lemma}

Finally, let $C\subset{\F}_n$ be a curve in $|L_{2,0}|$ disjoint from $\Sigma$ (not necessarily smooth nor irreducible). 
We let 
\begin{equation}\label{eqn: HE invol}
\iota_C :  {\F}_n \to {\F}_n 
\end{equation}
be the involution of ${\F}_n$ which on each $\pi$-fiber $F$ 
exchanges the two points $C|_F$ (or fixes $C|_F$ when they coincide) 
and fixes the one point $\Sigma|_F$. 
This extends the hyperelliptic involution of $C$. 
The fixed locus of $\iota_C$ is written as $H+\Sigma$ for a smooth $H\in|L_{1,0}|$. 
We thus have the ${\aut}({\F}_n)$-equivariant map 
\begin{equation}\label{average}
\varphi : |L_{2,0}|\dashrightarrow |L_{1,0}|, \quad C\mapsto H, 
\end{equation}
which will be used repeatedly in this article. 
The section $H$ must pass through the singular points of $C$. 
If we normalize $H$ to be $H_0$, 
the involution $\iota_C$ is given by $(x_3, y_3)\mapsto(x_3, -y_3)$ in the coordinate. 
Therefore, we have $\varphi(C)=H_0$ if and only if the equation $\sum_{i=0}^{2}f_i(x_3)y_3^{2-i}=0$ of $C$ satisfies $f_1\equiv0$.

%%%%%%%%%%%%%%
% Section : Eisenstein K3 
%%%%%%%%%%%%%%

\section{Eisenstein $K3$ surfaces}\label{sec: EK3}

\subsection{Eisenstein $K3$ surfaces}\label{ssec: EK3}

Let $X$ be a complex $K3$ surface with an automorphism group $G\subset{\aut}(X)$ of order $3$ 
which acts on $H^0(K_X)$ faithfully. 
We shall call such a pair $(X, G)$ an \textit{Eisenstein $K3$ surface}. 
We first review the basic theory of Eisenstein $K3$ surfaces 
following \cite{A-S}, \cite{Ta} and \cite{A-S-T}. 
Let 
\begin{equation}\label{eqn: inv lattice}
L(X, G) = H^2(X, {\Z})^G
\end{equation}
be the lattice of $G$-invariant cycles, and let 
\begin{equation}\label{eqn: anti-inv lattice}
E(X, G) = L(X, G)^{\perp}\cap H^2(X, {\Z})
\end{equation}
be its orthogonal complement. 
The presence of $G$ automatically implies that $X$ is algebraic, 
so that $L(X, G)$ is a hyperbolic lattice. 
We shall denote by $r$ the rank of $L(X, G)$. 
By the relation \eqref{eqn: anti-inv lattice}, 
the discriminant forms of $L(X, G)$ and $E(X, G)$ are canonically anti-isometric (\cite{Ni0}): 
\begin{equation}\label{eqn: isom disc form}
(A_{L(X, G)}, q_{L(X, G)}) \simeq (A_{E(X, G)}, -q_{E(X, G)}). 
\end{equation}
By \cite{A-S}, \cite{Ta} these discriminant groups are 3-elementary, namely 
$A_{L(X, G)}\simeq({\Z}/3{\Z})^a$ for some $a\geq0$. 

By the definition, the group $G$ acts on $E(X, G)$ with no non-zero invariant vector. 
Therefore, by choosing the distinguished generator $\rho\in G$ acting on $H^0(K_X)$ by $e^{2\pi i/3}$, 
the even lattice $E(X, G)$ is canonically endowed with the structure of an Eisenstein lattice 
in the sense of \S \ref{ssec: E lattice}. 
Moreover, since $G$ acts on $L(X, G)$ trivially, 
it acts on $A_{E(X, G)}$ trivially by \eqref{eqn: isom disc form}. 
Our usage of the terminology ''Eisenstein $K3$ surface'' comes from 
the viewpoint that $E(X, G)$ plays a fundamental role in the theory of such $K3$ surfaces. 

Artebani-Sarti \cite{A-S} and the third-named author \cite{Ta} classified Eisenstein $K3$ surfaces 
in terms of the pair $(r, a)$.

\begin{proposition}[\cite{A-S}, \cite{Ta}]\label{fixed locus}
The fixed locus $X^G$ of an Eisenstein $K3$ surface $(X, G)$ is of the form 
\begin{equation*}\label{eqn: fixed locus}
X^G = C^g \sqcup F_1 \sqcup\cdots\sqcup F_k\sqcup \{ p_1,\cdots, p_n\}
\end{equation*}
where $C^g$ is a genus $g$ curve, $F_i$ are $(-2)$-curves, and $p_j$ are isolated points with 
\begin{equation}\label{eqn: (g,k,n)}
g=\frac{22-r -2a}{4}, \qquad k=\frac{2+r-2a}{4}, \qquad n=\frac{r-2}{2}. 
\end{equation}
In the case $(r, a)=(8, 7)$ for which $(g, k)=(0, -1)$, 
this means fixed locus consisting of $3$ isolated points and no curve component. 
\end{proposition}

\begin{theorem}[\cite{A-S}, \cite{Ta}]\label{AST classification} 
The deformation type of an Eisenstein $K3$ surface $(X, G)$ is determined by the invariant $(r, a)$. 
All possible $(r, a)$ are shown in Figure \ref{geography}. 
\end{theorem}

\begin{figure}[hbtp]
\begin{center}
\includegraphics[width=9cm]{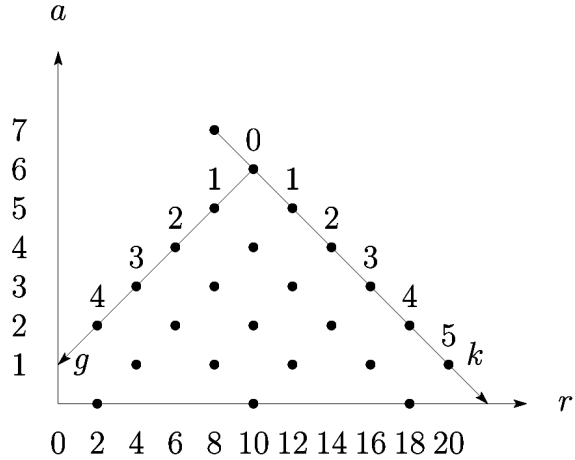} 
\caption{Distribution of invariants $(r,a)$} 
\label{geography}
\end{center}
\end{figure}

In other terms, Theorem \ref{AST classification} says that 
the deformation type of an Eisenstein $K3$ surface $(X, G)$ is determined by the Eisenstein lattice $E(X, G)$, 
which in turn is determined by the signature $(2, 20-r)$ and $a=l(A_{E(X, G)})$. 

\begin{lemma}[cf.~\cite{A-S}]\label{condition E lattice}
An indefinite Eisenstein lattice $(E, \rho)$ is isomorphic to $E(X, G)$ for an Eisenstein $K3$ surface $(X, G)$ if and only if 
$E$ can be primitively embedded into the $K3$ lattice $\Lambda_{K3}=U^3\oplus E_8^2$ as an even lattice, and $\rho$ acts trivially on $A_E$. 
\end{lemma}

\begin{proof}
Let $E\subset\Lambda_{K3}$ be such an Eisenstein lattice, 
which must have signature $(2, s)$ for some even number $s$. 
Let $L=E^{\perp}\cap\Lambda_{K3}$. 
By our assumption, $\rho$ extends to an isometry of $\Lambda_{K3}$ by acting trivially on $L$. 
We shall denote that extension also by $\rho$. 

Let $E\otimes{\C}=V\oplus\overline{V}$ be the eigendecomposition for $\rho$, 
where $\rho$ acts on $V$ by $e^{2\pi i/3}$. 
We choose a point ${\C}\omega\in{\proj}V$ such that $(\omega, \bar{\omega})>0$ 
and $(\omega, \delta)\ne0$ for any $(-2)$-vector $\delta\in E$. 
Since $(\omega, \omega)=0$, by the surjectivity of the period mapping 
we can find a $K3$ surface $X$ for which we have a Hodge isometry 
$\Phi\colon H^2(X, {\Z})\to (\Lambda_{K3}, {\C}\omega)$. 
Composing $\Phi$ with some reflections with respect to $(-2)$-curves on $X$, 
we may assume that $\Phi^{-1}(L)$ contains an ample class of $X$. 
Then by the Torelli theorem we have an automorphism $g$ of $X$ with $g^{\ast}=\Phi^{-1}\circ\rho\circ\Phi$. 
By the construction, $g$ is non-symplectic of order $3$ and we have a Hodge isometry 
$\Phi\colon E(X, \langle g\rangle)\to (E, {\C}\omega)$ preserving the Eisenstein structures. 
\end{proof}

By Theorem \ref{AST classification} and Lemma \ref{condition E lattice},  
the deformation types of Eisenstein $K3$ surfaces are in one-to-one correspondence with 
the isomorphism classes of Eisenstein lattices $E$ as in Lemma \ref{condition E lattice}, 
and Figure \ref{geography} may be regarded as classifying such Eisenstein lattices. 
Moreover, the proof of Lemma \ref{condition E lattice} tells that 
for two such Eisenstein lattices $E, E'\subset\Lambda_{K3}$ with the same invariant $(r, a)$, 
there exists an isometry $\gamma\in{\rm O}(\Lambda_{K3})$ such that 
$\gamma|_E$ gives an isomorphism $E\to E'$ of Eisenstein lattices.

Here we list concrete forms of the Eisenstein lattices $E$ for each fixed $g$: 

\begin{eqnarray*}
A_2(-1)\oplus A_2^{a-1},                               &  \qquad &  g=0   \\
U^2\oplus A_2^a,                                          & \qquad  &  g=1    \\
A_2(-1)\oplus A_2^{a-1}\oplus E_8,             &  \qquad  &  g=2     \\ 
U^2\oplus A_2^a\oplus E_8,                        &  \qquad  &  g=3     \\
A_2(-1)\oplus A_2^{a-1}\oplus E_8^2,          &  \qquad  &  g=4    \\
U^2\oplus E_8^2,                                          &  \qquad  & g=5 
\end{eqnarray*}

Next we study a relationship between the invariant lattice $L(X, G)$ and the fixed locus $X^G$. 
Let $\hat{X}\rightarrow X$ be the blow-up at the isolated fixed points $p_1,\cdots, p_n$ of $G$,   
and $E_i\subset\hat{X}$ the $(-1)$-curve over $p_i$. 
The $G$-action extends to $\hat{X}$ with the fixed locus 
\begin{equation*}
\hat{X}^G = C^g + F_1 + \cdots + F_k + E_1 + \cdots + E_n.  
\end{equation*}
We shall denote $L(\hat{X}, G)=H^2(\hat{X}, {\Z})^G$, 
which is freely generated by $L(X, G)$ and $E_1,\cdots, E_n$. 
Since $\hat{X}^G$ is a curve, the quotient surface $\hat{Y}=\hat{X}/G$ is smooth. 
It is easy to see that $\hat{Y}$ is rational. 
Let $\hat{f}\colon\hat{X}\to\hat{Y}$ be the quotient morphism. 
Substituting the relation $K_{\hat{X}}\sim \sum_iE_i$ into the ramification formula for $\hat{f}$, 
we obtain 
\begin{equation}\label{eqn:relation in L/NS_Y}
-\hat{f}^{\ast}K_{\hat{Y}} \sim 2C^g+2\sum_{i=1}^{k}F_i+\sum_{j=1}^{n} E_j, 
\end{equation}
which we regard as a relation among the curves $C^g, F_i, E_j$ in 
$L(\hat{X}, G)/\hat{f}^{\ast}NS_{\hat{Y}}$.

\begin{proposition}\label{gene gene L}
The invariant lattice $L(\hat{X}, G)$ is generated by the sublattice $\hat{f}^{\ast}NS_{\hat{Y}}$ 
and the classes of the fixed curves $C^g, F_i, E_j$. 
\end{proposition}

\begin{proof}
First note that $\hat{f}^{\ast}NS_{\hat{Y}}$ is of finite index in $L(\hat{X}, G)$, 
because for any $l\in L(\hat{X}, G)$ we have 
$3l=\hat{f}^{\ast}\hat{f}_{\ast}l \in \hat{f}^{\ast}NS_{\hat{Y}}$. 
Both $L(\hat{X}, G)$ and $\hat{f}^{\ast}NS_{\hat{Y}}\simeq NS_{\hat{Y}}(3)$ 
have $3$-elementary discriminant groups of length $a$, ${\rm rk}(NS_{\hat{Y}})$ respectively. 
Since 
\begin{equation*}
{\rm rk}(NS_{\hat{Y}})={\rm rk}(L(\hat{X}, G))=r+n, 
\end{equation*}
the sublattice $\hat{f}^{\ast}NS_{\hat{Y}}$ is of index $3^{(r+n-a)/2}$ in $L(\hat{X}, G)$.  
We have $\frac{r+n-a}{2}=k+n$ by \eqref{eqn: (g,k,n)}, so that the assertion reduces to the following lemma. 
\end{proof}

\begin{lemma}\label{1kodake}
Up to $\pm1$, \eqref{eqn:relation in L/NS_Y} is the only relation among $\{ C^{g}, F_i, E_j \}_{i,j}$
in the vector space $L(\hat{X}, G)/\hat{f}^{\ast }NS_{\hat{Y}}$ over ${\Z}/3{\Z}$.
\end{lemma}

\begin{proof}
Let 
\begin{equation}\label{eqn: relation I}
\alpha C^g + \sum_i\beta_iF_i + \sum_j\gamma_jE_j \equiv 0, \qquad \alpha, \beta_i, \gamma_j \in {\Z}/3{\Z}, 
\end{equation}
be a relation among $C^g, F_i, E_j$ in $L(\hat{X}, G)/\hat{f}^{\ast }NS_{\hat{Y}}$. 
Since $\hat{f}_{\ast}\hat{f}^{\ast}NS_{\hat{Y}}=3NS_{\hat{Y}}$, 
we apply $\hat{f}_{\ast}$ to \eqref{eqn: relation I} to obtain 
\begin{equation}\label{eqn: relation II}
\alpha \hat{f}_{\ast}C^g + \sum_i\beta_i\hat{f}_{\ast}F_i + \sum_j\gamma_j\hat{f}_{\ast}E_j \equiv 0 \quad \textrm{in} \: \:  NS_{\hat{Y}}/3NS_{\hat{Y}}. 
\end{equation}
We can identify 
$NS_{\hat{Y}}/3NS_{\hat{Y}}$ with $H_2(\hat{Y}, {\Z}/3{\Z})$ 
by the Poincar\'e duality and the universal coefficient theorem. 
Therefore \eqref{eqn: relation II} gives an element of the kernel of the map 
\begin{equation}\label{pushforward fixed curve}
\hat{f}_{\ast} : H_2(\hat{X}^G, {\Z}/3{\Z}) \to H_2(\hat{Y}, {\Z}/3{\Z}). 
\end{equation}
Regarding $\hat{X}^G$ as a curve on $\hat{Y}$ naturally, 
\eqref{pushforward fixed curve} fits into the homology exact sequence for the pair $(\hat{Y}, \hat{X}^G)$: 
\begin{equation*}
\cdots \to H_3(\hat{Y}, \hat{X}^G, {\Z}/3{\Z}) \to H_2(\hat{X}^G, {\Z}/3{\Z}) \stackrel{\hat{f}_{\ast}}{\to} H_2(\hat{Y}, {\Z}/3{\Z}) \to \cdots . 
\end{equation*}
Then we have $h_3(\hat{Y}, \hat{X}^G, {\Z}/3{\Z})=1$ by \cite{A-S-T} Lemma 2.5. 
This proves our claim. 
\end{proof}

%%%%%%%%%%%%%%%%%%%%%%%%%%%%%%
% subsection: Moduli spaces
%%%%%%%%%%%%%%%%%%%%%%%%%%%%%%

\subsection{Moduli spaces}\label{ssec: classification}

Let $(r, a)$ be an invariant in Figure \ref{geography}. 
We fix an Eisenstein lattice $(E, \rho)$ of signature $(2, 20-r)$ such that 
$A_E\simeq({\Z}/3{\Z})^a$ and that $\rho$ acts on $A_E$ trivially. 
Let $E\otimes{\C}=V\oplus\overline{V}$ be the eigendecomposition for $\rho$ 
where $\rho|_V=e^{2\pi i/3}$. 
The Hermitian form on $V$ defined by $(v, \bar{w})$ for $v, w\in V$,  
is isometric to $E\otimes{\R}$ up to a scaling (\S \ref{ssec: E lattice}) 
and thus has signature $(1, 10-r/2)$. 
Therefore the domain 
\begin{equation}\label{eqn: ball}
\mathcal{B}_{E} = \{ {\C}\omega\in{\proj}V, \; (\omega, \bar{\omega})>0 \} 
\end{equation}
is a complex ball of dimension $10-r/2$. 
The unitary group ${\rm U}(E)$ of $E$ acts on $\mathcal{B}_E$. 
We define a complex analytic divisor $\mathcal{H}$ in $\mathcal{B}_E$ by $\mathcal{H}=\sum_{\delta}\delta^{\perp}$ 
where $\delta$ range over $(-2)$-vectors in $E$. 
Then we consider the open set of the ball quotient (or Picard modular variety) 
\begin{equation}\label{def moduli} 
{\moduli} = {\rm U}(E) \backslash (\mathcal{B}_E - \mathcal{H}),  
\end{equation}
which is a normal quasi-projective variety of dimension $10-r/2$.

Let $(X, G)$ be an Eisenstein $K3$ surface of invariant $(r, a)$. 
By Theorem \ref{AST classification} there exists an isomorphism $\Phi\colon E(X, G)\to E$ of Eisenstein lattices. 
The ${\C}$-linear extension of $\Phi$, also denoted by $\Phi$, 
maps $H^{2,0}(X)$ to a point of $\mathcal{B}_E$.  
Then $\Phi(H^{2,0}(X))$ is contained in the complement of $\mathcal{H}$ 
(cf.~\cite{D-K0}, \cite{A-S-T}), 
and we define the period of $(X, G)$ by 
\begin{equation}\label{eqn: def of period}
\mathcal{P}(X, G)=[\Phi(H^{2,0}(X))]\in{\moduli}.
\end{equation}
This is independent of the choice of $\Phi$. 

\begin{theorem}[cf.~\cite{A-S-T}, \cite{D-K0}]\label{moduli}
The variety ${\moduli}$ is the moduli space of Eisenstein $K3$ surfaces of type $(r, a)$ 
in the following sense. 
 
$(1)$ For any family $(\mathcal{X}\to U, G)$ of such Eisenstein $K3$ surfaces over a variety $U$, 
the period map $\mathcal{P}\colon U\to{\moduli}$ is a morphism of varieties.  

$(2)$ Via the period mapping the points of ${\moduli}$ are in one-to-one correspondence with 
the isomorphism classes of such Eisenstein $K3$ surfaces. 
\end{theorem}
 
\begin{proof}
The fact that period maps are morphisms is a consequence of Borel's extension theorem \cite{Bo}. 
The surjectivity of the period mapping is proved in \cite{D-K0} \S 11 and also in \cite{A-S-T} 
(cf.~the proof of Lemma \ref{condition E lattice}). 
Here we shall supplement the proof of the injectivity, 
which is more or less asserted in \cite{A-S-T} \S 9 without proof. 
%(Some idea can also be found in \cite{D-K0} \S 11.) 
%Note that the description in \cite{A-S-T} of the monodromy group is replaced here by the arithmetic group ${\rm U}(E)$. 
Let us begin with the following basic lemma. 

\begin{lemma}\label{Weyl action nef}
Let $(X, G)$ be an Eisenstein $K3$ surface and 
let $W(X)$ be the Weyl group of $NS_X$ generated by $(-2)$-reflections. 
For every $l\in L(X, G)$ with $(l, l)\geq0$ there exists $w\in W(X)$ commuting with the $G$-action 
such that either $w(l)$ or $-w(l)$ is nef. 
\end{lemma}

\begin{proof}
This is analogous to \cite{B-H-P-V} Proposition VIII 21.1. 
We may assume that $(l, h_0)\geq0$ for an ample class $h_0\in NS_X$. 
Let $D\subset X$ be a $(-2)$-curve with $(l, D)<0$. 
Then for a generator $\rho\in G$ we have $(D, \rho(D))\leq0$. 
If not, the effective divisor class $C=D+\rho(D)+\rho^{-1}(D)$ in $L(X, G)$ 
would have norm $\geq0$ and satisfy $(l, C)<0$, which is a contradiction. 
Therefore $D$ is either preserved by $G$ or disjoint from $\rho(D)$. 
In the former case we apply to $l$ the reflection with respect to $D$, which commutes with the $G$-action. 
In the latter case the three curves $D$, $\rho(D)$ and $\rho^{-1}(D)$ are pairwise disjoint. 
Then we apply to $l$ the composition of the three reflections with respect to these curves, 
which also commutes with the $G$-action. 
As in \cite{B-H-P-V}, this process will terminate and $l$ will be finally mapped to a nef class. 
\end{proof}

Returning to the proof of Theorem \ref{moduli}, 
we let two Eisenstein $K3$ surfaces $(X, G), (X', G')$ of type $(r, a)$ have the same period in ${\moduli}$. 
This means that we have an isomorphism $\gamma:E(X, G)\to E(X', G')$ of Eisenstein lattices 
preserving the Hodge structures. 
We want to extend $\gamma$ to a Hodge isometry $\Phi\colon H^2(X, {\Z})\to H^2(X', {\Z})$. 
Since $L(X, G)$ and $L(X', G')$ are isometric, 
by a standard argument of discriminant group (cf. \cite{Ni0}) 
it suffices to show that the natural homomorphism 
${\rm O}(L(X, G))\to{\rm O}(A_{L(X, G)})$ is surjective. 
When $(r, a)\ne(2, 2), (4, 3), (8, 7)$, we have $r\geq a+2$ so that our claim follows from \cite{Ni0} Theorem 1.14.2. 
The case $(r, a)=(2, 2)$ is easily checked. 
For the remaining two cases, we may resort to the assertions (i), (iii) of the Theorem of \cite{M-M}.  
Thus we obtain a desired extension $\Phi$ of $\gamma$. 
By the above lemma we may compose $\Phi$ with a $G$-equivariant $w\in W(X)$ so that 
$\Phi\circ w$ preserves the ample cones. 
By the Torelli thorem we have an isomorphism $\varphi\colon X'\to X$ with $\varphi^{\ast}=\Phi\circ w$. 
Then $\varphi$ is ${\Z}/3{\Z}$-equivariant because $\varphi^{\ast}$ is so. 
Therefore $(X, G)$ is isomorphic to $(X', G')$.  
\end{proof}

We set $g=(22-r-2a)/4$ as in \eqref{eqn: (g,k,n)}. 
Let $\mathcal{M}_g$ be the moduli space of genus $g$ curves. 
When $g>0$, we have the \textit{fixed curve map} 
\begin{equation}\label{eqn: def fixed curve map}
{\moduli} \to \mathcal{M}_g, \qquad (X, G)\mapsto C^g, 
\end{equation}
where $C^g$ is the genus $g$ curve in $X^G$. 
This map will be analyzed for some ${\moduli}$ in the rest of the article.

%%%%%%%%%%%%%%%%%%%%%%%%%%%%%%%
%%%  Subsection: Discriminant Cover 
%%%%%%%%%%%%%%%%%%%%%%%%%%%%%%%

\subsection{Marked Eisenstein $K3$ surfaces}\label{ssec: cover}

We define a Galois cover of ${\moduli}$ that will be used 
in our degree calculation of period maps (\S \ref{ssec: recipe}). 
It is also treated systematically in \cite{D-K0} \S 11.  
Let $E$ be the Eisenstein lattice used in the definition \eqref{def moduli} of ${\moduli}$. 
The natural homomorphism ${\rm U}(E) \to {\rm O}(A_E)$ is surjective 
by Corollary \ref{surj II} (for $(r, a)\ne(8, 5), (10, 4)$) and 
Propositions \ref{birat (8,5)}, \ref{birat (10,4)} (for $(r, a)=(8, 5), (10, 4)$ respectively).  
Let $\widetilde{{\rm U}}(E)$ be the kernel of ${\rm U}(E) \to {\rm O}(A_E)$. 
We consider the ball quotient   
\begin{equation}\label{def cover}
{\cover} = \widetilde{{\rm U}}(E) \backslash \mathcal{B}_E. 
\end{equation}
Its open set over ${\moduli}$ is a Galois cover of ${\moduli}$  
with Galois group ${\rm O}(A_E)/\pm1$. 
In particular, the degree of the projection ${\cover}\dashrightarrow{\moduli}$ is given by 
\begin{equation}\label{proj degree}
\left\{ \begin{array}{cl} 
                             |{\rm O}(A_E)|/2,  &    \quad        a>0,    \\
                                  1,                                          &    \quad        a=0.    \\
                               \end{array} \right. 
\end{equation}
Since $(A_E, q_E)$ is a finite quadratic form in characteristic $3$, 
we can calculate $|{\rm O}(A_E)|$ by referring to, e.g., \cite{Atlas}. 
We shall use the following standard notation for orthogonal groups in characteristic $3$: 
${\rm GO}(2m+1, 3)$, ${\rm GO}^+(2m, 3)$ and ${\rm GO}^-(2m, 3)$.

As essentially explained in \cite{D-K0} \S 10 -- \S 11,  
${\cover}$ is birationally a moduli space of Eisenstein $K3$ surfaces with marking of its invariant lattice. 
We fix an even hyperbolic $3$-elementary lattice $L$ of rank $r$ and $l(A_L)=a$, 
a primitive embedding $L\subset\Lambda_{K3}$, 
and an isometry $E\simeq L^{\perp}\cap\Lambda_{K3}$ of quadratic forms. 
We extend the ${\Z}/3{\Z}$-action on $E$ to $\Lambda_{K3}$ by the trivial action on $L$. 
Suppose that we are given an Eisenstein $K3$ surface $(X, G)$ with 
an isometry $j\colon L\to L(X, G)$ of quadratic forms. 
By the surjectivity of ${\rm U}(E) \to {\rm O}(A_E)$,\footnote{
For \S \ref{ssec:(8,5)} and \S \ref{ssec:(10,4)}: 
If the surjectivity of ${\rm U}(E) \to {\rm O}(A_E)$ is yet uncertain at this moment, one should consider only those $((X, G), j)$ such that 
$j$ \textit{can be} ${\Z}/3{\Z}$-equivariantly extended to $\Lambda_{K3}\to H^2(X, {\Z})$.
In this case, the Galois group of ${\cover}\dashrightarrow{\moduli}$ is a priori just a subgroup of ${\rm O}(A_E)/\pm1$ 
(but in fact the whole ${\rm O}(A_E)/\pm1$).}\label{footnote1}  
the embedding $j$ extends to a ${\Z}/3{\Z}$-equivariant isometry 
$\Phi\colon\Lambda_{K3}\to H^2(X, {\Z})$. 
Since the restriction of $\Phi$ to $L$ is fixed, 
the isometry $\Phi|_E\colon E\to E(X, G)$ %for the orthogonal complements 
is determined up to the action of $\widetilde{{\rm U}}(E)$ by \cite{Ni0}. 
Then we define the period of the Eisenstein $K3$ surface $(X, G)$ with the lattice-marking $j$ by 
\begin{equation}\label{lifted period}
\widetilde{\mathcal{P}}((X, G), j) = [ \Phi|_E^{-1}(H^{2,0}(X)) ] \in {\cover}. 
\end{equation}
Clearly, two such lattice-marked Eisenstein $K3$ surfaces $((X, G), j)$, $((X', G'), j')$  
have the same $\widetilde{\mathcal{P}}$-period in ${\cover}$ 
if and only if 
there exists a ${\Z}/3{\Z}$-equivariant Hodge isometry 
$\Psi \colon H^2(X, {\Z}) \to H^2(X', {\Z})$ 
with $\Psi\circ j = j'$. 
The open set of ${\cover}$ over ${\moduli}$ parametrizes 
such equivalence classes of Eisenstein $K3$ surfaces with lattice-marking. 

%As in the involution case \cite{Ma1}, 
%there exists a finite surjective morphism ${\cover}\to{\cove}_{r, a'}$ for $a>a'$. 
%One may exploit these isogenies to prove the unirationality of most ${\moduli}$,  
%but the cases to be analyzed do not reduce so much. 
%In this article we shall use ${\cover}$ rather as a tool for the calculation of degree of certain period maps.  

%%%%%%%%%%%%%%%%%%%%%%%%%%
%%     Triple cover construction
%%%%%%%%%%%%%%%%%%%%%%%%%%

\section{Triple cover construction}

\subsection{Mixed branch}\label{ssec:mixed branch} 

We develop triple cover construction of Eisenstein $K3$ surfaces 
in a moderate generality sufficient for the proof of Theorem \ref{main}. 
We propose the notion of \textit{mixed branch} as an analogue of DPN pair \cite{A-N}, 
that is, singular branch \textit{curve} on smooth surface. 
The key idea is to distinguish 
the branch components turning to isolated fixed points 
from those components turning to fixed curves by multiplicity of divisor.  
The formality of the resolution process \eqref{resol process} works keeping this geometric idea. 

\begin{definition}\label{def:mixed branch}
Let $Y$ be a smooth rational surface. 
A mixed branch on $Y$ is 
a ${\Q}$-divisor $B=B_1+\frac{1}{2}B_2$ linearly equivalent to $-\frac{3}{2}K_Y$, 
where $B_1, B_2$ are reduced curves having no common component, 
with the following properties. 

$(1)$ ${\rm Sing}(B_1)$ are at most nodes, cusps, tacnodes and ramphoid cusps. 

$(2)$ $B_2$ is a union of rational curves, 
and its singularities (if any) are only ordinary triple points disjoint from ${\rm Sing}(B_1)$. 

$(3)$ If $B_2$ passes through a singular point $p$ of $B_1$, 
then $p$ is a node or cusp of $B_1$, and $B_1+B_2$ has more than one tangent at $p$. 
\end{definition}

We call $\frac{1}{2}B_2$ the \textit{shadow part}\footnote{
One might draw $\frac{1}{2}B_2$ as a half-transparent curve.}  
of $B$.
The condition $(1)$ comes from the demand that 
the local triple cover around $p\in{\rm Sing}(B_1)$ branched over $B_1$ has only A-D-E singularities 
(see the next \S \ref{ssec:pure branch}). 
Let us denote $(B_i)_{sm}=B_i\backslash{\rm Sing}(B_i)$. 
The multiplicity of $B$ at a singular point $p$ of $B_1+B_2$ is classified as follows: 
\begin{itemize}
\item $3/2$ \; ($p\in{\rm Sing}(B_2)\backslash B_1$ or $p\in(B_1)_{sm}\cap(B_2)_{sm}$) 
\item $2$  \; ($p\in{\rm Sing}(B_1)\backslash B_2$) 
\item $5/2$ \; ($p\in{\rm Sing}(B_2)\cap B_1$ or $p\in{\rm Sing}(B_1)\cap B_2$) 
\end{itemize}

We can resolve a mixed branch $B=B_1+\frac{1}{2}B_2$ in the following way. 
Let $Y'\to Y$ be the blow-up at a singular point $p$ of $B_1+B_2$. 
We define a mixed branch on $Y'$ by%\footnote{
%Compare this blow-up formula with the one in the order $2$ case, e.g., \cite{A-N} p. 21.}  
\begin{equation}\label{resol process}
B_1'+\frac{1}{2}B_2' = 
\widetilde{B}_1+\frac{1}{2}\widetilde{B}_2+(m-\frac{3}{2})E, 
\end{equation}
where $\widetilde{B}_i$ is the strict transform of $B_i$, 
$m$ is the multiplicity of $B$ at $p$, 
and $E$ is the $(-1)$-curve over $p$.  
One checks that $B'=B_1'+\frac{1}{2}B_2'$ is linearly equivalent to $-\frac{3}{2}K_{Y'}$ 
and satisfies the conditions $(1)$--$(3)$ in Definition \ref{def:mixed branch}. 
Continuing this resolution process $\cdots\to(Y'', B'')\to(Y', B')$, 
we finally obtain a mixed branch $(\hat{Y}, \hat{B}_1+\frac{1}{2}\hat{B}_2)$ with 
$\hat{B}_1+\hat{B}_2$ smooth. 
We shall call this procedure the \textit{right resolution} of $(Y, B)$.  
Substituting the relation $2\hat{B}_1+\hat{B}_2 \sim -3K_{\hat{Y}}$ into the adjunction formula, 
we see that every rational component of $\hat{B}_1$ (resp. $\hat{B}_2$) is 
a $(-6)$-curve (resp. $(-3)$-curve). 
Since $\hat{B}_1-\hat{B}_2 \sim 3(K_{\hat{Y}}+\hat{B}_1)$, 
we can take a cyclic triple cover $\hat{f}\colon\hat{X}\to\hat{Y}$ 
branched over $\hat{B}_1+\hat{B}_2$ by the following general lemma. 

\begin{lemma}
Let $Y$ be a complex manifold and $D_1, D_2$ be disjoint smooth divisors on $Y$ 
with $D_1-D_2\in d {\Pic}(Y)$. 
Then there exists a cyclic cover $X\to Y$ of degree $d$ branched over $D_1+D_2$. 
\end{lemma}

\begin{proof}
As usual, we choose a line bundle $L$ with an isomorphism $L^{\otimes d}\simeq {\sheaf}_Y(D_1-D_2)$. 
We compactify the total space of $L$ to 
$\overline{L}={\proj}({\sheaf}_Y\oplus L)$ (adding $\infty$ to each fiber). 
If $s$ is a meromorphic section of $L^{\otimes d}$ with ${\rm div}(s)=D_1-D_2$, 
then the divisor 
$\{ v\in \overline{L}, v^{\otimes d}=s\}$ in $\overline{L}$ gives the desired covering. 
\end{proof}

%By a similar argument as in \cite{B-H-P-V} Lemma I.17.1, we have 
%\begin{equation*}
%\hat{f}^{\ast}(K_{\hat{Y}}+\hat{B}_1)  \sim  
%\hat{f}^{-1}(\hat{B}_1)-\hat{f}^{-1}(\hat{B}_2), 
%\end{equation*}
%where $\hat{f}^{-1}(\hat{B}_1)$ denotes the reduced inverse image. 

Alternatively, 
by the relation $2\hat{B}_1+\hat{B}_2 \in {\rm Pic}(\hat{Y})$ 
we can take a cyclic triple cover $\hat{X}'\to\hat{Y}$ 
branched over $2\hat{B}_1+\hat{B}_2$. 
This $\hat{X}'$ has cuspidal singularities along $\hat{B}_1$, 
and $\hat{X}$ can also be obtained as the normalization of $\hat{X}'$. 

By the ramification formula we see that  
\begin{equation*}
K_{\hat{X}} 
\sim \hat{f}^{\ast}(K_{\hat{Y}} + \hat{B}_1+\hat{B}_2) 
- \hat{f}^{-1}(\hat{B}_1+\hat{B}_2) 
\sim \hat{f}^{-1}(\hat{B}_2),  
\end{equation*}
where $\hat{f}^{-1}(\hat{B}_i)$ denotes the reduced inverse image. 
The divisor $\hat{f}^{-1}(\hat{B}_2)$ is a disjoint union of $(-1)$-curves. 
Blowing them down, we obtain a surface $X$ with $K_X\simeq{\sheaf}_X$, namely a $K3$ or abelian surface. 
The ${\Z}/3{\Z}$-action on $\hat{X}\to\hat{Y}$ equips $X$ 
with a non-symplectic symmetry $G$ of order $3$. 
The abelian case does happen, but is quite rare.  
Specifically, 

\begin{lemma}\label{abelian case}
The surface $X$ is abelian if and only if $B_1=0$ and $B_2$ has nine components. 
\end{lemma}

\begin{proof}
If $X$ is abelian, the fixed locus $X^G$ is either 
the union of isolated points or of disjoint elliptic curves (cf.~\cite{B-L}). 
In the latter case the quotient $X/G$ is again an abelian surface, which is out of the present situation. 
In the former case we have $|X^G|=9$ by \cite{B-L} Example 13.2.7, 
and thus $B_2$ has nine components and $B_1$ is empty. 
Conversely, if $B_1=0$ and $B_2$ has nine components, 
$X$ cannot be $K3$ by Figure \ref{geography}. 
\end{proof}

When $X$ is a $K3$ surface,  
we thus obtain an Eisenstein $K3$ surface associated to the mixed branch $(Y, B_1+\frac{1}{2}B_2)$. 

Let $E\subset Y$ be one of the following types of $(-1)$-curves: 
\begin{itemize}
\item those $E$ transverse to $B_1+B_2$; 
\item components $E$ of $B_1$ with $(E, B_2)=1$; 
\item components $E$ of $B_2$ which are disjoint from other components of $B_2$. 
\end{itemize} 
If $\pi\colon Y\to \overline{Y}$ is the blow-down of $E$, 
then $(\overline{Y}, \pi(B_1)+\frac{1}{2}\pi(B_2))$ is again a mixed branch. 
%Thus, with the blow-up rule \eqref{resol process}, 
%our notion of mixed branch allows flexibility of birational transformation to some extent. 
%But it is not preserved by every blow-down, due to the conditions on singularity: 
%this is a defect of our (tentative) notion. 
%We could also extend it by allowing any blown-down image of 
%smooth mixed branch (cf.~\S \ref{ssec:(8,5)}), but with less effectivity at present. 
In this way, by composing blow-up \eqref{resol process} and such blow-down, 
we can pass from a given mixed branch to another one with common smooth model. 
Regrettably we have restriction on the type of blow-down, 
due to the singularity conditions in Definition \ref{def:mixed branch}. 
For that we could also extend the definition of mixed branch 
by allowing any blown-down image of smooth mixed branch (cf.~\S \ref{ssec:(8,5)}), 
but with less effectivity at present. 
Anyway, the present generality is handy, and sufficient for giving canonical construction of general members of most ${\moduli}$. 

Actually, for seventeen ${\moduli}$ we will use mixed branch with no shadow. 
Thus in the next subsection we shall be more specific in that case.

\begin{remark}
We were led to the notion of mixed branch 
by tracking resolution of $-\frac{3}{2}K_{{\F}_n}$-curves on ${\F}_n$ (see \S \ref{ssec:pure branch}). 
It seems that the rule \eqref{resol process} would also explain the resolution process in \cite{O-T} for certain singular del Pezzo surfaces,  
by detecting the shadow part $B_2$ by discrepancy. 
\end{remark}

\subsection{Anti-tri-halfcanonical curves on Hirzebruch surfaces}\label{ssec:pure branch}

A mixed branch with no shadow is just a reduced curve 
$B\sim-\frac{3}{2}K_Y$ with at most nodes, cusps, tacnodes and ramphoid cusps as the singularities. 
Since $3K_Y\in2{\Pic}(Y)$ and $|\!-\!\frac{3}{2}K_Y|$ contains a reduced member, 
$Y$ must be a Hirzebruch surface ${\F}_n$ with $n\in\{0, 2, 4, 6\}$. 
In this case, we have $B\in3{\Pic}({\F}_n)$ so that 
we may take a cyclic triple cover $\overline{X}\to{\F}_n$ branched over $B$. 
Looking at the local equations of the singularities of $B$, 
we see that the singularities of $\overline{X}$ (lying over ${\rm Sing}(B)$) are as follows: 
\begin{itemize}
\item $A_2$-points ($z^3=x^2+y^2$) over nodes ($x^2+y^2=0$), 
\item $D_4$-points ($z^3=x^2+y^3$) over cusps ($x^2+y^3=0$), 
\item $E_6$-points ($z^3=x^2+y^4$) over tacnodes ($x^2+y^4=0$), 
\item $E_8$-points ($z^3=x^2+y^5$) over ramphoid cusps ($x^2+y^5=0$). 
\end{itemize}
In particular, $\overline{X}$ has only A-D-E singularities. 
Since $K_{\overline{X}}\sim\mathcal{O}_{\overline{X}}$, 
we can resolve ${\rm Sing}(\overline{X})$ to obtain a $K3$ surface $X$ with a non-symplectic symmetry $G$ of order $3$. 
($X$ cannot be an abelian surface by Lemma \ref{abelian case}.) 
It is clear that this Eisenstein $K3$ surface $(X, G)$ coincides with 
the one obtained in \S \ref{ssec:mixed branch} using resolution of $B$. 
A virtue in the present situation is that we have a natural projection $f\colon X\to{\F}_n$. 

Let $L\in{\Pic}({\F}_n)$ be the bundle $L_{1,0}$ (resp. ${\sheaf}_{{\F}_0}(1, 1)$) 
when $n=2, 4, 6$ (resp. $n=0$). 
%Since the projection formula $H^0(f^{\ast}L)\simeq H^0(f_{\ast}{\sheaf}_X\otimes L)$ is ${\Z}/3{\Z}$-equivariant 
%where ${\Z}/3{\Z}$ acts on $f^{\ast}L$ and $f_{\ast}{\sheaf}_X$ naturally, 
The subspace $f^{\ast}H^0(L)\subset H^0(f^{\ast}L)$ is the eigenspace for $G$ with eigenvalue $1$. 
The morphism $X\to f^{\ast}|L|^{\vee}$ associated to the linear system $f^{\ast}|L|$ is 
the composition of $f$ and the morphism ${\F}_n\to|L|^{\vee}$ associated to $L$. 
The last one is the contraction of the $(-n)$-curve $\Sigma$ (resp. an embedding) when $n\geq2$ (resp. $n=0$). 
Checking that $f^{\ast}|L|\subset|f^{\ast}L|$ has strictly larger dimension 
than the other two eigenspaces, 
we have the following useful 

\begin{lemma}\label{recovery}
Let $B, B'\in|\!-\!\frac{3}{2}K_{{\F}_n}|$ be as above, 
and $(X, G), (X', G')$ be the associated Eisenstein $K3$ surfaces 
with the projections $f\colon X\to{\F}_n, f'\colon X'\to{\F}_n$. 
If we have an isomorphism $\varphi\colon(X, G)\to(X', G')$ with 
$\varphi^{\ast}(f')^{\ast}L\simeq f^{\ast}L$, 
then we have an automorphism $\psi$ of  ${\F}_n$ with $f'\circ\varphi=\psi\circ f$. 
\end{lemma}

Let us describe the configurations of curves lying over the singularities of $B$. 
Let $(\hat{Y}, \hat{B}_1+\frac{1}{2}\hat{B}_2)$ be the right resolution of $({\F}_n, B)$
and $\hat{X}\to\hat{Y}$ be the triple cover branched over $\hat{B}_1+\hat{B}_2$. 
Let $p$ be a singular point of $B$. 
Following the blow-up procedure \eqref{resol process}, 
we see that the dual graph of the curves on $\hat{Y}$ contracted to $p$ is, 
according to the type of singularity, as follows.

\begin{picture}(100,25)\label{graph A_2}%A_2
	\thicklines	
	%Name
	\put(10,-5){{\large $A_2$}}	
	% Dots
	\put(60,0){\circle{8}}
	\put(75,-3){$\bigstar$}
	\put(100,0){\circle{8}}	
	% Lines
	\put(64,0){\line(1,0){12}}
	\put(84,0){\line(1,0){12}}	
\end{picture}

\begin{picture}(100,25)\label{graph D_4}%D_4 
	\thicklines
	%Name
	\put(10,-5){{\large $D_4$}}	
	% Dots
	\put(60,0){\circle*{8}}
	\put(80,0){\circle{8}}
	\put(95,-3){$\bigstar$}	
	% Lines
	\put(64,0){\line(1,0){12}}
	\put(84,0){\line(1,0){12}}	
\end{picture}

\begin{picture}(100,55)\label{graph E_6}%E_6 
	\thicklines
	%Name
	\put(10,-5){{\large $E_6$}}	
	% Dots
	\put(60,0){\circle{8}}
	\put(75,-3){$\bigstar$}        
	\put(100,0){\circle{8}}
	\put(120,0){\circle{8}}  \put(120,0){\circle{4}}       
                                      \put(120,20){\circle{8}}     \put(114.9,35){$\bigstar$}     
	\put(140,0){\circle{8}}        
     \put(155,-3){$\bigstar$}  
     \put(180,0){\circle{8}}        
	% Lines
	\put(64,0){\line(1,0){12}}
	\put(84,0){\line(1,0){12}}
	\put(104,0){\line(1,0){12}}  
	\put(120,4){\line(0,1){12}}  \put(120,24){\line(0,1){12}}	
	\put(124,0){\line(1,0){12}}
	\put(144,0){\line(1,0){12}}
	\put(164,0){\line(1,0){12}}		
\end{picture}

\begin{picture}(100,60)\label{graph E_8}%E_8 
	\thicklines
	%Name
	\put(10,-5){{\large $E_8$}}	
	% Dots
	\put(55,-3){$\bigstar$}
	\put(80,0){\circle{8}}        
	\put(100,0){\circle{8}}  \put(100,0){\circle{4}}
	\put(120,0){\circle{8}}
	\put(135,-3){$\bigstar$}
	\put(160,0){\circle{8}}
	\put(180,0){\circle{8}}  \put(180,0){\circle{4}}
	\put(200,0){\circle{8}}
	\put(215,-3){$\bigstar$}
	\put(240,0){\circle{8}}
	              \put(180,20){\circle{8}}  \put(174.9,35){$\bigstar$}       
      % Lines
	\put(64,0){\line(1,0){12}}
	\put(84,0){\line(1,0){12}}
	\put(104,0){\line(1,0){12}}  
	\put(124,0){\line(1,0){12}}
	\put(144,0){\line(1,0){12}}
	\put(164,0){\line(1,0){12}}	
	      \put(180,4){\line(0,1){12}}  \put(180,24){\line(0,1){12}}
	\put(184,0){\line(1,0){12}}
	\put(204,0){\line(1,0){12}}	                                
	\put(224,0){\line(1,0){12}}
\end{picture}

\begin{figure}[h]
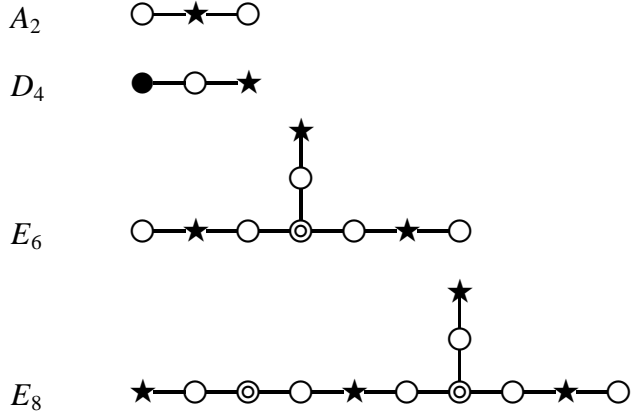

\caption{Dual graphs of exceptional curves on $\hat{Y}$}
\label{dual graph}
\end{figure}

%\vspace{7mm}

\noindent
Here 
a white circle represents a $(-1)$-curve; 
a black circle represents a $(-2)$-curve (disjoint from $\hat{B}_1+\hat{B}_2$); 
a double circle represents a $(-6)$-curve (a component of $\hat{B}_1$); 
and a star represents a $(-3)$-curve (a component of $\hat{B}_2$). 
The reduced inverse images of those curves by $\hat{X}\to\hat{Y}$ are respectively 
a $(-3)$-curve; three disjoint $(-2)$-curves; a $(-2)$-curve; and a $(-1)$-curve. 
Blowing-down the last $(-1)$-curves, 
we obtain the configuration of exceptional curves of the resolution $X\to\overline{X}$ over $p$. 
Its dual graph $\Gamma_p$ (isomorphic to the Dynkin graph of $A_2$-, $D_4$-, $E_6$- or $E_8$-type) is obtained from the graph in Figure \ref{dual graph} 
by multiplying the black circle thrice and contracting the stars. 
Thus the stars turn to isolated fixed points of $G$, and the double circles turn to fixed curves. 
When $p$ is a cusp, $G$ acts on $\Gamma_p$ by the cyclic permutations; 
in other cases $G$ acts on $\Gamma_p$ trivially. 

One should note that, when $p$ is a node or tacnode, 
there are \textit{two} identifications of our geometric dual graph $\Gamma_p$ with the abstract $A_2$- or $E_6$-graph. 
A choice of such an identification corresponds to a labeling for the two branches of $B$ at $p$. 
On the other hand, when $p$ is a ramphoid cusp, such an identification is unique.

From these we can compute the topological invariants of $(X, G)$ as follows. 
Let $k_0+1$ be the number of components of $B$, and let $a_2, d_4, e_6$ and $e_8$ denote 
the number of nodes, cusps, tacnodes, and ramphoid cusps of $B$ respectively. 
Then the number $k+1$ of fixed curves of $(X, G)$ is given by 
\begin{equation*}\label{calculate k 3.3}
k = k_0 + e_6 + 2e_8, 
\end{equation*}
and the number $n$ of isolated fixed points of $(X, G)$ is given by
\begin{equation*}\label{calculate n 3.3}
n = a_2 + d_4 + 3e_6 + 4e_8. 
\end{equation*}
The rank $r$ of the invariant lattice $L(X, G)$ is the Picard number of $\hat{Y}$ minus $n$, 
which is given by 
\begin{equation*}\label{calculate r 3.3}
r = 2 + 2a_2 + 2d_4 + 6e_6 + 8e_8. 
\end{equation*}

In the rest of this subsection we work under the following ``genericity" assumption: 
\begin{equation}\label{genericity assumption}
\textit{${\rm Sing}(B)$ does not contain cusps.}
\end{equation}
Then for a singular point $p\in B$, 
we denote by $\Lambda_p\subset NS_X$ the root lattice generated by the exceptional curves of 
the resolution $X\to\overline{X}$ over $p$. 
As observed above, $\Lambda_p$ is contained in the invariant lattice $L(X, G)$. 
Let $B=\sum_{i=0}^{k_0}B_i$ be the irreducible decomposition of $B$, 
and $F_i\subset X$ be the fixed curve of $G$ with $f(F_i)=B_i$.

\begin{proposition}\label{gene L}
The invariant lattice $L(X, G)$ is generated by 
the sublattice $f^{\ast}NS_{{\F}_n}\oplus(\oplus_p\Lambda_p)$ where $p\in{\rm Sing}(B)$, 
and the classes of $F_i$, $0\leq i\leq k_0$. 
\end{proposition}

\begin{proof}
Consider the blow-up $\pi: \hat{X}\to X$ of the isolated fixed points. 
By Proposition \ref{gene gene L} and Figure \ref{dual graph},
the invariant lattice $L(\hat{X}, G)$ of $(\hat{X}, G)$ is generated by 
$\pi^{\ast}(f^{\ast}NS_{{\F}_n}\oplus(\oplus_p\Lambda_p))$, the classes of $\pi^{\ast}F_i$, 
and the classes of exceptional curves of $\pi$. 
Contracting the exceptional curves, we see our assertion for $L(X, G)$. 
\end{proof}

Let us emphasize (again) that when $p$ is a ramphoid cusp, we have a unique isometry $E_8\to\Lambda_p$ that maps the natural root basis to the classes of $(-2)$-curves, 
while when $p$ is a node (resp. tacnode), we have two such natural isometries $A_2\to\Lambda_p$ (resp. $E_6\to\Lambda_p$) 
corresponding to the two labelings of the branches of $B$ at $p$.

Finally, we shall construct an ample class in $L(X, G)$ using the above objects. 
We denote by $e_{i\pm}, e_i$ the root basis of the $E_6$- and $E_8$-lattices 
according to the following numberings for the vertices of the $E_6$- and $E_8$-graphs:

\begin{picture}(100,50)\label{numbering} 
	\thicklines	
	
%E6	
% Dots
	\put(30,0){\circle*{4}}        
	\put(50,0){\circle*{4}}
	\put(70,0){\circle*{4}}        
     \put(70,20){\circle*{4}}           
	\put(90,0){\circle*{4}}        
     \put(110,0){\circle*{4}}        
	% Lines
	\put(30,0){\line(1,0){20}}
	\put(50,0){\line(1,0){20}}
	\put(70,0){\line(0,1){20}}	
	\put(70,0){\line(1,0){20}}
	\put(90,0){\line(1,0){20}}		
	% Labels
	\put(31,-11){1+}
	\put(51,-11){2+}
	\put(71,-11){3}
	\put(73,11){4}
      \put(91,-11){2-}
      \put(111,-11){1-}

%E8	
	% Dots
	\put(170,0){\circle*{4}}        
	\put(190,0){\circle*{4}}
	\put(210,0){\circle*{4}}        
	\put(230,0){\circle*{4}}
	\put(250,0){\circle*{4}}        
     \put(250,20){\circle*{4}}           
	\put(270,0){\circle*{4}}        
     \put(290,0){\circle*{4}}        
	% Lines
	\put(170,0){\line(1,0){20}}
	\put(190,0){\line(1,0){20}}
	\put(210,0){\line(1,0){20}}
	\put(230,0){\line(1,0){20}}
	\put(250,0){\line(0,1){20}}	
	\put(250,0){\line(1,0){20}}
	\put(270,0){\line(1,0){20}}		
	% Labels
	\put(171,-11){7}
	\put(191,-11){6}
	\put(211,-11){5}
	\put(231,-11){4}
	\put(251,-11){3}
	\put(253,11){8}
      \put(271,-11){2}
      \put(291,-11){1}
\end{picture}

\vspace{1cm}

For a tacnode $p\in{\rm Sing}(B)$, 
let $D_p\in\Lambda_p$ be the image of $e_3+\sum_{i=1}^{2}3^{3-i}(e_{i+}+e_{i-})$ 
by either of the natural isometries $E_6\to\Lambda_p$; 
For a ramphoid cusp $p\in{\rm Sing}(B)$, 
let $D_p\in\Lambda_p$ be the image of $\sum_{i=1}^{6}3^{6-i}e_i$ 
by the natural isometry $E_8\to\Lambda_p$.

\begin{lemma}\label{ample}
For an arbitrary ample line bundle $L\in{\rm Pic}({\F}_n)$, the class 
\begin{equation*}
3^{20}f^{\ast}L + 3^{10}\sum_{i=0}^{k_0}F_i + \sum_{p}D_p, 
\end{equation*}
where $p$ run over tacnodes and ramphoid cusps of $B$, 
is ample.  
\end{lemma}

\begin{proof}
Check the Nakai criterion (see, e.g., \cite{B-H-P-V} Chapter IV.6). 
\end{proof}

\subsection{Degree of period map}\label{ssec: recipe}

As in \S \ref{ssec:pure branch}, let ${\F}_n$ be a Hirzebruch surface with $n\in\{0, 2, 4, 6\}$. 
Suppose we have an irreducible, ${\aut}({\F}_n)$-invariant locus 
$U\subset|-\frac{3}{2}K_{{\F}_n}|$ such that 
$({\rm i})$ every member $B_u\in U$ has only nodes, tacnodes and ramphoid cusps as the singularities, and
$({\rm ii})$ the number of singularities of $B_u$ of each type and the number of components of $B_u$ are constant. 
Then the Eisenstein $K3$ surfaces associated to $({\F}_n, B_u)$ have constant invariant $(r, a)$, 
and we obtain a period map $p\colon U\to\mathcal{M}_{r, a}$ as a morphism of varieties. 
Since this construction is invariant under ${\aut}({\F}_n)$, 
the morphism $p$ descends to a rational map 
\begin{equation}\label{period map}
\mathcal{P} : U/{\aut}({\F}_n) \dashrightarrow \mathcal{M}_{r, a}. 
\end{equation}
Here $U/{\aut}({\F}_n)$ stands for a rational quotient, i.e., 
an arbitrary model of the invariant field ${\C}(U)^{{\aut}({\F}_n)}$. 
In this subsection we shall explain a systematic method to calculate the degree of $\mathcal{P}$, 
which is a fundamental in this article. 
It is parallel to the one in the involution case \cite{Ma2}, though some points need to be modified.

%Below we provide a recipe of calculation, which is to be applied repeatedly in the rest of this article. 
%There we will present the materials and sufficient instructions in each case, and then leave the details which may be worked out following the recipe. 
%We will avoid repetition of similar argument in this way. 

We use the Galois cover ${\cover}$ of ${\moduli}$ defined in \eqref{def cover}. 
Recall that an open set of ${\cover}$ parametrizes the equivalence classes of 
lattice-marked Eisenstein $K3$ surfaces $((X, G), j)$, 
where $j$ is a marking of the invariant lattice $L(X, G)$ by some reference lattice $L$. 
For the calculation of $\deg(\mathcal{P})$, 
we define a certain cover $\widetilde{U}$ of $U$ and construct a generically injective lift 
\begin{equation*}
{\lift} : \widetilde{U}/{\aut}({\F}_n)_0 \dashrightarrow {\cover} 
\end{equation*}
of $\mathcal{P}$, 
where ${\aut}({\F}_n)_0$ is the identity component of ${\aut}({\F}_n)$. 
%(Note that ${\aut}({\F}_n)_0={\aut}({\F}_n)$ unless $n=0$.) 
We then compare the two projections 
$\widetilde{U}/{\aut}({\F}_n)_0\dashrightarrow U/{\aut}({\F}_n)$ and 
${\cover}\dashrightarrow{\moduli}$. 
More precisely, 

\begin{itemize}
\item we define a cover $\widetilde{U} \to U$ parametrizing curves $B_u\in U$ endowed with 
          reasonable labelings $\mu$ of the singularities, the branches at nodes and tacnodes, and the components. 
\item Proposition \ref{gene L} implies an appropriate definition of the reference lattice $L$. 
          Then for each $(B_u, \mu)\in\widetilde{U}$, the labeling $\mu$ naturally induces a lattice-marking 
          $j\colon L\to L(X, G)$ for the Eisenstein $K3$ surface $(X, G)=p(B_u)$. 
          Considering the period of $((X, G), j)$ as defined in \eqref{lifted period}, 
          we obtain a lift $\tilde{p}\colon\widetilde{U}\to{\cover}$ of $p$. 
\item We check that $\tilde{p}$ is invariant under ${\aut}({\F}_n)_0$, which acts trivially on $NS_{{\F}_n}$.  
          Thus $\tilde{p}$ descends to a rational map ${\lift}\colon\widetilde{U}/{\aut}({\F}_n)_0 \dashrightarrow {\cover}$ 
          which is a lift of $\mathcal{P}$. 
\item We show that ${\lift}$ is generically injective by proving that the $\tilde{p}$-fibers are ${\aut}({\F}_n)_0$-orbits. 
          If two $(B_u, \mu), (B_{u'}, \mu')\in\widetilde{U}$ have the same $\tilde{p}$-period, 
          we have a ${\Z}/3{\Z}$-equivariant Hodge isometry $\Phi \colon H^2(X', {\Z}) \to H^2(X, {\Z})$ preserving 
          the lattice-markings for the associated Eisenstein $K3$ surfaces. 
          Then $\Phi$ preserves the ample cones by Lemma \ref{ample}, 
          so that we obtain an isomorphism $\varphi\colon X\to X'$ with $\varphi^{\ast}=\Phi$ by the Torelli theorem. 
          The isomorphism $\varphi$ is ${\Z}/3{\Z}$-equivariant because $\varphi^{\ast}$ is so. 
          Using Lemma \ref{recovery}, we see that $\varphi$ induces an automorphism $\psi$ of ${\F}_n$ with 
          $\psi\circ f=f'\circ\varphi$, where $f\colon X\to{\F}_n$, $f'\colon X'\to{\F}_n$ are the natural projections. 
          Then $\psi$ acts trivially on $NS_{{\F}_n}$ and maps $(B_u, \mu)$ to $(B_{u'}, \mu')$. 
          This verifies our assertion. 
\item Now assume that $U/{\aut}({\F}_n)$ has the same dimension as ${\moduli}$. 
          Since ${\cover}$ is irreducible, ${\lift}$ is then birational. 
          Therefore $\deg(\mathcal{P})$ is equal to \eqref{proj degree} 
          divided by the degree of the projection $\widetilde{U}/{\aut}({\F}_n)_0\dashrightarrow U/{\aut}({\F}_n)$. 
          The latter may be calculated by geometric consideration. 
\end{itemize}

We shall exhibit typical examples that illustrate 
how this recipe actually works and how one should define $\widetilde{U}$ and ${\lift}$, 
which is left ambiguous in the above explanation. 
In the rest of the article the recipe will be applied over and over. 
To avoid repetition we will leave the detail of argument there, 
which can be worked out by referring to the examples below as models.

\begin{example}\label{ex1}
We consider curves on the Hirzebruch surface ${\F}_6$. 
Let $U\subset|L_{2,0}|$ be the locus of irreducible curves having three nodes and no other singularity. 
For $C\in U$ we associate the $-\frac{3}{2}K_{{\F}_6}$-curve $C+\Sigma$. 
By the triple cover construction this defines an Eisenstein $K3$ surface $(X, G)$ of invariant $(g, k)=(2, 1)$, 
and we obtain a period map $\mathcal{P}\colon U/{\aut}({\F}_6) \dashrightarrow \mathcal{M}_{8,3}$. 

Let $f\colon X\to{\F}_6$ be the natural projection. 
By Proposition \ref{gene L} the invariant lattice $L(X, G)$ is generated by 
$f^{\ast}NS_{{\F}_6}\simeq U(3)$, 
three copies of the $A_2$-lattice obtained from the nodes of $C$, 
and the classes of fixed curves. 
In view of this, we shall define a reference lattice $L$ as follows. 
Let $M$ be the lattice $U(3)\oplus A_2^3$ with a natural basis $\{ u, v, e_{1+}, e_{1-},\cdots, e_{3-}\}$, 
where $\{ u, v\}$ are basis of $U(3)$ with $(u, u)=(v, v)=0$ and $(u, v)=3$, 
and $\{ e_{i+}, e_{i-}\}$ are root basis of the $i$-th $A_2$-lattice with $(e_{i+}, e_{i-})=1$. 
We define vectors $f_0, f_1\in M^{\vee}$ by 
$3f_0=2(u+3v)-3\sum_{i=1}^3(e_{i+}+e_{i-})$ and 
$3f_1=u-3v$. 
Then let $L$ be the overlattice $L=\langle M, f_0, f_1\rangle$, 
which is even and 3-elementary of invariant $(r, a)=(8, 3)$. 

In order to calculate $\deg(\mathcal{P})$, 
for $C\in U$ we first distinguish its three nodes, 
and then the two branches at each node. 
This is realized by an $\frak{S}_3\ltimes(\frak{S}_2)^3$-cover $\widetilde{U}\to U$. 
Explicitly, $\widetilde{U}$ may be defined as the locus in $U\times({\proj}T{\F}_6)^6$ 
of those $(C, v_{1+}, v_{1-},\cdots, v_{3-})$ such that 
$v_{i+}$ and $v_{i-}$ are the two tangents of $C$ at a node, say $p_i$, 
and that ${\rm Sing}C=\{ p_1, p_2, p_3\}$. 
This labels the nodes and the branches at them compatibly. 
Accordingly, we denote by $E_{i\pm}\subset X$ the $(-2)$-curve lying over the infinitely near point $v_{i\pm}$ of $p_i$. 
Then $E_{i+}$ and $E_{i-}$ form a root basis of the $A_2$-lattice over $p_i$. 
The fixed curve of $(X, G)$ is decomposed as $F_0+F_1$ such that 
$F_0$ (resp. $F_1$) is the component with $f(F_0)=C$ (resp. $f(F_1)=\Sigma$). 
Then we have a natural isometry $j\colon L\to L(X, G)$ by sending 
$j(a(u+3v)+bv)=f^{\ast}L_{a,b}$, 
$j(e_{i\pm})=[E_{i\pm}]$, and $j(f_i)=[F_i]$. 
In this way we associate a lattice-marked Eisenstein $K3$ surface $((X, G), j)$ to $(C, v_{i\pm})$. 
This defines a morphism $\tilde{p}\colon\widetilde{U}\to{\cove}_{8,3}$, 
which descends to a lift ${\lift}\colon\widetilde{U}/{\aut}({\F}_6) \dashrightarrow {\cove}_{8,3}$ of $\mathcal{P}$ 
because ${\aut}({\F}_6)$ acts trivially on $NS_{{\F}_6}$. 

We shall show that the $\tilde{p}$-fibers are ${\aut}({\F}_6)$-orbits. 
If $\tilde{p}(C, v_{i\pm})=\tilde{p}(C', v_{i\pm}')$ for two $(C, v_{i\pm}), (C', v_{i\pm}') \in \widetilde{U}$, 
there exists a ${\Z}/3{\Z}$-equivariant Hodge isometry $\Phi \colon H^2(X', {\Z}) \to H^2(X, {\Z})$ 
with $\Phi\circ j'=j$ for the associated $((X, G), j)$ and $((X', G'), j')$. 
By Lemma \ref{ample} and the Torelli theorem 
we obtain an isomorphism $\varphi\colon X\to X'$ with $\varphi^{\ast}=\Phi$. 
The last equality implies that 
$\varphi^{\ast}G'=G$, $\varphi(E_{i\pm})=E_{i\pm}'$, and $\varphi^{\ast}((f')^{\ast}L_{a,b})=f^{\ast}L_{a,b}$, 
where $f$, $E_{i\pm}$ (resp. $f'$, $E_{i\pm}'$) are the objects constructed from 
$(C, v_{i\pm})$ (resp. $(C', v_{i\pm}')$) as above. 
Then by Lemma \ref{recovery} we obtain an automorphism $\psi$ of ${\F}_6$ with 
$f'\circ\varphi=\psi\circ f$. 
This shows that $\psi(v_{i\pm})=v_{i\pm}'$. 
We also have $\psi(C)=C'$ because $\psi$ maps the branch curve of $f$ to that of $f'$. 
This proves our assertion, and hence ${\lift}$ is generically injective. 
Since ${\dim}(U/{\aut}({\F}_6))=6$, 
${\lift}$ is actually birational. 

Finally, we compare the two projections 
$\widetilde{U}/{\aut}({\F}_6)\dashrightarrow U/{\aut}({\F}_6)$ and 
${\cove}_{8,3}\dashrightarrow\mathcal{M}_{8,3}$. 
The latter has degree $|{\rm O}(A_L)|/2$, where $|{\rm O}(A_L)|=|{\rm GO}(3, 3)|=2^3\cdot3!$ by \cite{Atlas}. 
On the other hand, 
the stabilizer in ${\aut}({\F}_6)$ of a general $C\in U$ is generated by its hyperelliptic involution $\iota_C$ defined in \eqref{eqn: HE invol}. 
It follows that $\widetilde{U}/{\aut}({\F}_6)\dashrightarrow U/{\aut}({\F}_6)$ has degree 
$|\frak{S}_3\ltimes(\frak{S}_2)^3|/2$. 
Therefore $\mathcal{P}$ is birational. 
\end{example}

\begin{example}\label{ex2}
We consider curves on ${\F}_2$. 
Let $U\subset|L_{2,0}|\times|L_{0,2}|$ be the locus of pairs $(C, D)$ where 
$C$ and $D=D_1+D_2$ are smooth and transverse to each other. 
We consider the six-nodal $-\frac{3}{2}K_{{\F}_2}$-curves $C+D+\Sigma$ 
to obtain Eisenstein $K3$ surfaces of invariant $(g, k)=(1, 3)$. 
This defines a period map $\mathcal{P}\colon U/{\aut}({\F}_2) \dashrightarrow \mathcal{M}_{14,2}$. 

We prepare a reference lattice $L$ as follows. 
Let $M$ be the lattice $U(3)\oplus A_2^6$ with a natural basis $\{ u, v, e_{1+}, e_{1-},\cdots, e_{6-}\}$ 
defined in the same way as Example \ref{ex1}. 
We define vectors $f_0,\cdots, f_3\in M^{\vee}$ by 
$3f_0  =  2(u+v)-\sum_{i=1}^{4}(2e_{i-}+e_{i+})$,  
$3f_1  =   v-\sum_{i=1}^3(2e_{(2i-1)+}+e_{(2i-1)-})$,  
$3f_2  =  v-\sum_{i=1}^3(2e_{(2i)+}+e_{(2i)-})$, and      
$3f_3  =  u-v-\sum_{i=5}^{6}(2e_{i-}+e_{i+})$. 
Then the overlattice $L=\langle M, f_0,\cdots, f_3\rangle$ is 
even and 3-elementary of invariant $(r, a)=(14, 2)$. 

For the calculation of $\deg(\mathcal{P})$, 
we first distinguish the two components of $D$, 
and then the intersection points of each component with $C$. 
Specifically, we consider the locus $\widetilde{U}\subset U\times({\F}_2)^4$ of those $(C, D, p_1,\cdots, p_4)$ such that 
$\{ p_i\}_{i=1}^4=C\cap D$ and that $p_1, p_3$ lie on the same component of $D$. 
We accordingly denote by $D_1$ (resp. $D_2$) the component of $D$ through $p_1, p_3$ (resp. $p_2, p_4$). 
Thus the components of $D$ and the four nodes $C\cap D$ are labelled compatibly. 
The projection $\widetilde{U}\to U$ is an $\frak{S}_2\ltimes(\frak{S}_2)^2$-covering. 
The remaining data for $C+D+\Sigma$ are labelled automatically: 
we denote $p_5=D_1\cap\Sigma$; $p_6=D_2\cap\Sigma$; 
$v_{i+}$ the tangent of $D$ at $p_i$; and $v_{i-}$ the tangent of $C+\Sigma$ at $p_i$. 
In this way we obtain a complete labeling for $C+D+\Sigma$. 
Then let $(X, G)=\mathcal{P}(C, D)$ and $f\colon X\to{\F}_2$ be the natural projection. 
We denote by $E_{i\pm}\subset X$ the $(-2)$-curve lying over the infinitely near point $v_{i\pm}$ of $p_i$. 
The fixed curve for $(X, G)$ is decomposed as $F_0+\cdots+F_3$ such that 
$f(F_0)=C$, $f(F_i)=D_i$ for $i=1, 2$, and $f(F_3)=\Sigma$. 
As before, we have an isometry $j\colon L\to L(X, G)$ by 
$j(a(u+v)+bv)=f^{\ast}L_{a,b}$, 
$j(e_{i\pm})=[E_{i\pm}]$, and $j(f_i)=[F_i]$. 
Considering the period of $((X, G), j)$, 
we obtain a lift ${\lift}\colon\widetilde{U}/{\aut}({\F}_2) \dashrightarrow {\cove}_{14,2}$ of $\mathcal{P}$. 
By a similar argument as in Example \ref{ex1}, 
we see that ${\lift}$ is generically injective. 
Since ${\dim}(U/{\aut}({\F}_2))=3$, ${\lift}$ is then birational. 

The projection ${\cove}_{14,2}\dashrightarrow\mathcal{M}_{14,2}$ has degree $|{\rm O}(A_L)|/2$. 
Since $L$ is isometric to $U\oplus E_8\oplus A_2^2$, we have $|{\rm O}(A_L)|=2^3$ by a direct calculation. 
On the other hand, a general $(C, D)\in U$ has no nontrivial stabilizer in ${\aut}({\F}_2)$ other than the hyperelliptic involution $\iota_C$ of $C$. 
Hence the projection $\widetilde{U}/{\aut}({\F}_2) \dashrightarrow U/{\aut}({\F}_2)$ has degree $4$, 
and so the map $\mathcal{P}$ is birational. 
\end{example}

\begin{example}\label{ex3}
Our recipe for $-\frac{3}{2}K_{{\F}_n}$-curves may also be utilized 
for some general mixed branches, 
via a birational transformation. 
As an illustrative example, 
let $U\subset|{\Oplane}(4)|\times|{\Oplane}(1)|$ be the open set of pairs $(C, L)$ such that $C$ is a smooth quartic transverse to the line $L$. 
We regard $(C, L)$ as a mixed branch $C+\frac{1}{2}L$ on ${\proj}^2$. 
By the resolution of $C+\frac{1}{2}L$, we obtain an Eisenstein $K3$ surface of invariant $(g, k)=(3, 0)$. 
This defines a period map $\mathcal{P}\colon U/{\PGL}_3\to\mathcal{M}_{4,3}$. 

To calculate $\deg(\mathcal{P})$, 
let $\widetilde{U}$ be the locus in $U\times({\proj}^2)^4$ of those $(C, L, p_1,\cdots, p_4)$ such that $C\cap L=\{ p_i\}_{i=1}^4$. 
The space $\widetilde{U}$ is an $\frak{S}_4$-cover of $U$ parametrizing mixed branches $C+\frac{1}{2}L$ 
endowed with labelings of the four intersection points $C\cap L$. 
We want to show that $\mathcal{P}$ lifts to a birational map $\widetilde{U}/{\PGL}_3\to{\cove}_{4,3}$. 
For that we blow-up $p_1, p_2$ and then blow-down (the strict transform of) $L$. 
This transforms $C+\frac{1}{2}L$ to a one-nodal curve $C^{\dag}$ of bidegree $(3, 3)$ on $Q={\proj}^1\times{\proj}^1$.  
The two branches of $C^{\dag}$ at its node are distinguished by the labeling $(p_3, p_4)$, and 
the two rulings on $Q$ are distinguished by the labeling $(p_1, p_2)$. 
Specifically, we assign the $i$-th projection $Q\to{\proj}^1$ to the pencil of lines through $p_i$. 
Conversely, given a general one-nodal $C^{\dag}\in|\mathcal{O}_Q(3, 3)|$, 
we blow-up $Q$ at $p={\rm Sing}(C^{\dag})$ and then blow-down the two ruling fibers $F_1, F_2$ through $p$ to 
obtain a smooth plane quartic $C$. 
Let $L\subset{\proj}^2$ be the image of the $(-1)$-curve over $p$. 
Among the four points $C\cap L$, two correspond to the two branches of $C^{\dag}$ at $p$, 
and the rest two are given by $F_i\cap C^{\dag}\backslash p$. 
Hence the four points $C\cap L$ are labelled after one distinguishes  
the two branches of $C^{\dag}$ and the two rulings on $Q$ respectively. 
Summing up, if $V\subset|{\sheaf}_Q(3, 3)|$ is the locus of one-nodal curves 
and $\widetilde{V}\to V$ is the double cover labeling the branches at nodes, 
we have a natural birational identification 
$\widetilde{U}/{\PGL}_3\sim\widetilde{V}/({\PGL}_2)^2$. 
Here $({\PGL}_2)^2$ is the identity component of ${\aut}(Q)$ preserving the two rulings. 
Now we may apply our recipe to $\widetilde{V}$ to obtain a birational map 
$\widetilde{V}/({\PGL}_2)^2\dashrightarrow{\cove}_{4,3}$. 
This gives a desired lift $\widetilde{U}/{\PGL}_3\to{\cove}_{4,3}$ of $\mathcal{P}$. 

The quotient $\widetilde{U}/{\PGL}_3$ is an $\frak{S}_4$-cover of $U/{\PGL}_3$, 
while the Galois group of ${\cove}_{4,3}\to\mathcal{M}_{4,3}$ is ${\rm O}(A_L)/\pm1$ for the lattice $L=U(3)\oplus A_2$. 
We have $|{\rm O}(A_L)|=|{\rm GO}(3, 3)|=2\cdot4!$ by \cite{Atlas}. 
Therefore $\mathcal{P}$ is birational. 
\end{example}

\begin{remark}\label{variant recipe}
In Example \ref{ex3}, we could also apply a variant of the recipe \textit{directly} to the mixed branches $C+\frac{1}{2}L$. 
Indeed, a labeling of the four points $C\cap L$ defines a marking of the blown-up invariant lattice $L(\hat{X}, G)$, 
which induces that of $L(X, G)$. 
The lattice $L(X, G)$ encodes all the relevant geometric informations: 
(i) the $G$-invariant rational map $f\colon X\dashrightarrow {\proj}^2$ can be recovered from the line bundle $f^{\ast}{\Oplane}(1)$, which is free of degree $4$; 
and (ii) every point of $C\cap L$ is the image by $f$ of a $(-2)$-curve on $X$ preserved by $G$. 
\end{remark}

\begin{remark}
A similar recipe is proposed in the involution case \cite{Ma2} 
for the degree calculation for double cover construction. 
It utilizes geometric labeling for the branch curves as well, 
but does not require to label the branches at double points. 
This is the main difference with the present recipe. 
\end{remark} 

%\begin{remark}
%The principle of the present recipe should directly apply to general mixed branch. 
%We here confine ourselves to the framework sufficient for the rest of this article. 
%\end{remark}

%%%%%%%%%%%%%%
% Section : g=5
%%%%%%%%%%%%%%

\section{The case $g=5$}\label{sec:g=5}

We begin the proof of Theorem \ref{main}. 
We first study the case $g=5$ using curves on the Hirzebruch surface ${\F}_6$. 
Let $U\subset|L_{2,0}|$ be the open set of smooth curves. 
By Lemma \ref{linear system} (3),  
$U/{\aut}({\F}_6)$ is identified with the moduli space $\mathcal{H}_5$ of hyperelliptic curves of genus $5$. 
For $C\in U$ we take the triple cover $X\to{\F}_6$ branched over the $-\frac{3}{2}K_{{\F}_6}$-curve $C+\Sigma$. 
This defines the period map $\mathcal{P}\colon\mathcal{H}_5\to\mathcal{M}_{2,0}$. 
Then $\mathcal{P}$ is injective because the fixed curve map \eqref{eqn: def fixed curve map} for $\mathcal{M}_{2,0}$ gives the left inverse. 
Since ${\dim}\mathcal{H}_5={\dim}\mathcal{M}_{2,0}$, then $\mathcal{P}$ is dominant (actually isomorphic). 
Katsylo \cite{Ka1} proved that $\mathcal{H}_5$ is rational. 
Summing up, 

\begin{proposition}\label{rational (2,0)}
The space $\mathcal{M}_{2,0}$ is naturally birational to $\mathcal{H}_5$ and thus is rational. 
\end{proposition}

%%%%%%%%%%%%%%
% Section : g=4
%%%%%%%%%%%%%%

\section{The case $g=4$}\label{sec:g=4}

In this section we study the case $g=4$. 
Kond\=o \cite{Ko} proved that $\mathcal{M}_{2,2}$ is birational to the moduli space of genus $4$ curves, 
which is proven to be rational by Shepherd-Barron \cite{SB}. 
Here we study the space $\mathcal{M}_{4,1}$. 

We consider curves on ${\F}_6$. 
Let $U\subset|L_{2,0}|$ be the locus of irreducible one-nodal curves. 
For $C\in U$ we take the triple cover of ${\F}_6$ branched over the nodal $-\frac{3}{2}K_{{\F}_6}$-curve $C+\Sigma$. 
This defines a period map $\mathcal{P}\colon U/{\aut}({\F}_6)\dashrightarrow\mathcal{M}_{4,1}$. 

\begin{proposition}\label{period map (4,1)}
The map $\mathcal{P}$ is birational. 
\end{proposition}

\begin{proof}
Let $\widetilde{U}\subset U\times({\proj}T{\F}_6)^2$ be the locus of $(C, v_1, v_2)$ 
such that $\{ v_1, v_2\}$ are the tangents of $C$ at its node. 
The space $\widetilde{U}$ is a double cover of $U$ labelling the branches at the nodes of $C$. 
As in Example \ref{ex1}, we will see that $\mathcal{P}$ lifts to a birational map 
$\widetilde{U}/{\aut}({\F}_6)\dashrightarrow{\cove}_{4,1}$. 
Since ${\rm O}(A_L)=\{\pm1\}$ for the invariant lattice $L=U\oplus A_2$, 
we actually have ${\cove}_{4,1}=\mathcal{M}_{4,1}$. 
On the other hand, we have $\widetilde{U}/{\aut}({\F}_6)=U/{\aut}({\F}_6)$ 
because the stabilizer in ${\aut}({\F}_6)$ of every $C\in U$ contains its hyperelliptic involution $\iota_C$ defined in \eqref{eqn: HE invol}, 
which exchanges the two branches of $C$ at its node. 
Therefore $\mathcal{P}$ has degree $1$. 
\end{proof}
 
\begin{proposition}\label{rational (4,1)}
The quotient $U/{\aut}({\F}_6)$ is rational. 
Therefore $\mathcal{M}_{4,1}$ is rational. 
\end{proposition}

\begin{proof}
We perform the elementary transformation at the node of $C\in U$, 
which transforms $C$ to a smooth curve $C^{\dag}\in|L_{2,0}|$ on ${\F}_5$. 
This induces the birational equivalence 
\begin{equation}\label{ele trans (4,1)}
U/{\aut}({\F}_6) \sim (|L_{2,0}|\times|L_{0,1}|)/{\aut}({\F}_5). 
\end{equation}
By the slice method (cf. \cite{Do}), the right side is birational to $|L_{2,0}|/G$ 
where $G\subset{\aut}({\F}_5)$ is the stabilizer of a point of $|L_{0,1}|\simeq\Sigma$. 
Then $G$ is connected and solvable by Lemma \ref{linear system} (1), 
and our assertion follows from Miyata's theorem \cite{Miy}. 
\end{proof}

By \eqref{ele trans (4,1)} and Lemma \ref{linear system} (3),  
we see that the fixed curve map \eqref{eqn: def fixed curve map} for $\mathcal{M}_{4,1}$ is a dominant morphism onto the hyperelliptic locus $\mathcal{H}_4$ 
whose general fibers are birationally identified with the hyperelliptic pencils.

%%%%%%%%%%%%%%
% Section : g=3
%%%%%%%%%%%%%%

\section{The case $g=3$}\label{sec:g=3}

\subsection{The rationality of $\mathcal{M}_{4,3}$}

Let $U\subset|{\Oplane}(4)|\times|{\Oplane}(1)|$ be the open set of pairs $(C, L)$ such that 
$C$ is smooth and transverse to $L$. 
We use the ${\Q}$-divisors $C+\frac{1}{2}L$ as mixed branches. 
The associated Eisenstein $K3$ surfaces have invariant $(g, k)=(3, 0)$. 
In Example \ref{ex3} we showed that the induced period map 
$U/{\PGL}_3\dashrightarrow\mathcal{M}_{4,3}$ 
is birational. 

\begin{proposition}\label{rational (4,3)}
The quotient $U/{\PGL}_3$ is rational. 
Therefore $\mathcal{M}_{4,3}$ is rational. 
\end{proposition}

\begin{proof}
Using the no-name lemma (cf. \cite{Do}) for the projection $|{\Oplane}(4)|\times|{\Oplane}(1)|\to|{\Oplane}(4)|$,  
we have $U/{\PGL}_3\sim{\proj}^2\times(|{\Oplane}(4)|/{\PGL}_3)$. 
The quotient $|{\Oplane}(4)|/{\PGL}_3$ is rational by Katsylo \cite{Ka2}. 
\end{proof}

Since $|{\Oplane}(4)|/{\PGL}_3$ is canonically birational to the moduli space $\mathcal{M}_3$ of genus $3$ curves, 
the fixed curve map $\mathcal{M}_{4,3}\to\mathcal{M}_3$ is dominant with 
general fibers birationally identified with the canonical systems.

\subsection{The rationality of $\mathcal{M}_{6,2}$}

We consider curves on ${\F}_6$. 
Let $U\subset|L_{2,0}|$ be the locus of irreducible two-nodal curves $C$. 
Taking the triple covers of ${\F}_6$ branched over $C+\Sigma$,  
we obtain a period map $\mathcal{P}\colon U/{\aut}({\F}_6)\dashrightarrow\mathcal{M}_{6,2}$. 

\begin{proposition}\label{period map (6,2)}
The map $\mathcal{P}$ is birational. 
\end{proposition}

\begin{proof}
Let $\widetilde{U}\subset U\times({\proj}T{\F}_6)^4$ be the locus of $(C, v_{11}, v_{12}, v_{21}, v_{22})$ 
such that $\{ v_{ij}\}_{i,j}$ are the tangents of $C$ at its nodes and that $v_{11}, v_{12}$ share the base points. 
By $\widetilde{U}$ the nodes and the branches at them are labelled compatibly. 
The projection $\widetilde{U}\to U$ is an $\frak{S}_2\ltimes(\frak{S}_2)^2$-covering. 
As in Example \ref{ex1}, $\mathcal{P}$ lifts to a birational map 
$\widetilde{U}/{\aut}({\F}_6)\dashrightarrow{\cove}_{6,2}$. 
Since the invariant lattice $L$ is isometric to $U\oplus A_2^2$, we have $|{\rm O}(A_L)/\pm1|=4$. 
On the other hand, a general $C\in U$ has no stabilizer other than its hyperelliptic involution $\iota_C$, 
which exchanges the two tangents at each node. 
Thus the projection 
$\widetilde{U}/{\aut}({\F}_6)\to U/{\aut}({\F}_6)$ has degree $2^{-1}\cdot2^3$. 
Therefore $\mathcal{P}$ is birational. 
\end{proof}
 
\begin{proposition}\label{rational (6,2)}
The quotient $U/{\aut}({\F}_6)$ is rational. 
Therefore $\mathcal{M}_{6,2}$ is rational. 
\end{proposition}

\begin{proof}
As in the proof of Proposition \ref{rational (4,1)}, we perform the elementary transformations at the nodes of $C\in U$.  
This induces the birational equivalence 
\begin{equation}\label{ele trans (6,2)}
U/{\aut}({\F}_6) \sim (|L_{2,0}|\times|L_{0,2}|)/{\aut}({\F}_4). 
\end{equation}

We consider the ${\aut}({\F}_4)$-equivariant map 
\begin{equation*}
\psi = (\varphi, {\rm id}) : |L_{2,0}|\times|L_{0,2}| \dashrightarrow |L_{1,0}|\times|L_{0,2}|, \quad (C, F_1+F_2) \mapsto (H, F_1+F_2), 
\end{equation*}
where $\varphi$ is as defined in \eqref{average}. 
By Lemma \ref{linear system} (2), 
${\aut}({\F}_4)$ acts on $|L_{1,0}|\times|L_{0,2}|$ almost transitively. 
We normalize $H$ to be $H_0$ in \S \ref{ssec: Hirze}, and $F_i$ to be $\{ x_i=0\}$. 
Then the stabilizer of $(H_0, F_1+F_2)$ is given by 
\begin{equation}\label{stabilizer section+two fibers}
G = \{ g_{\alpha,0}\}_{\alpha\in{\C}^{\times}} \times (\langle\iota\rangle\ltimes\{ h_{\beta}\}_{\beta\in {\C}^{\times}})
\simeq  {\C}^{\times}\times(\frak{S}_2\ltimes{\C}^{\times}), 
\end{equation}
where $g_{\alpha,0}$, $\iota$, $h_{\beta}$ are as defined in \eqref{$R$-action in coordinate}--\eqref{auto coord 3}. 
On the other hand, we identify $H^0(L_{2,0})$ with the linear space 
$\{ \sum_{i=0}^2 f_i(x_3)y_3^{2-i}\}$ as in \eqref{def eq}. 
Then, as explained in the end of \S \ref{ssec: Hirze}, 
the fiber $\psi^{-1}(H_0, F_1+F_2)=\varphi^{-1}(H_0)$ is an open set of the linear subspace 
${\proj}V\subset|L_{2,0}|$ defined by $f_1\equiv0$. 
By the slice method for $\psi$, we have 
\begin{equation*}
(|L_{2,0}|\times|L_{0,2}|)/{\aut}({\F}_4) \sim {\proj}V/G. 
\end{equation*}

We expand the polynomials $f_2$ as $f_2(x_3)=\sum_{j=0}^{8}a_jx_3^j$.  
The generators $g_{\alpha,0}, \iota, h_{\beta}$ of $G$ act on $V$ by 
\begin{equation}\label{action g_alpha}
g_{\alpha,0} : y_3^2 \mapsto \alpha^{-2}y_3^2, \quad x_3^j \mapsto x_3^j, 
\end{equation}
\begin{equation*}\label{action h_beta}
h_{\beta} : y_3^2 \mapsto \beta^{4}y_3^2, \quad x_3^j \mapsto \beta^{4-j}x_3^j, 
\end{equation*}
\begin{equation*}\label{action iota}
\iota : y_3^2 \mapsto y_3^2, \quad x_3^j \mapsto x_3^{8-j}. 
\end{equation*}
Thus the $G$-representation $V$ is decomposed as 
\begin{equation*}
V= {\C}y_3^2 \oplus  \mathop{\bigoplus}_{i=0}^{4} W_i, \qquad W_i={\C}\langle x_3^{4-i}, x_3^{4+i} \rangle. 
\end{equation*}
If we consider the subrepresentation $W=\oplus_{i=0}^{4}W_i$ and the subgroup $H=\langle\iota\rangle\ltimes\{ h_{\beta}\}_{\beta\in {\C}^{\times}}$, 
then ${\proj}V/G$ is birational to ${\proj}W/H$. 
We set $W'=W_1\oplus W_2$ and $W''=W_0\oplus W_3\oplus W_4$. 
The projection ${\proj}W-{\proj}W''\to{\proj}W'$ from $W''$ is an $H$-linearized vector bundle. 
Since $H$ acts on ${\proj}W'$ almost freely, 
we have ${\proj}W/H \sim{\C}^5\times({\proj}W'/H)$ by the no-name lemma. 
Then ${\proj}W'/H$ is rational because it is $2$-dimensional. 
\end{proof}

By \eqref{ele trans (6,2)} and Lemma \ref{linear system} (3),  
the fixed curve map for $\mathcal{M}_{6,2}$ is a dominant morphism to the hyperelliptic locus $\mathcal{H}_3$ 
whose general fibers are birationally identified with the canonical systems.

\subsection{The rationality of $\mathcal{M}_{8,1}$}

We consider curves on ${\F}_4$. 
Let $U\subset|L_{2,0}|\times|L_{0,1}|$ be the open set of those $(C, F)$ such that 
$C$ is smooth and transverse to $F$. 
For $(C, F)\in U$ we take the triple cover of ${\F}_4$ branched over 
the nodal $-\frac{3}{2}K_{{\F}_4}$-curve $C+F+\Sigma$.  
This defines a period map 
$\mathcal{P}\colon U/{\aut}({\F}_4)\dashrightarrow\mathcal{M}_{8,1}$. 

\begin{proposition}\label{period map (8,1)}
The map $\mathcal{P}$ is birational. 
\end{proposition}

\begin{proof}
We consider a double cover $\widetilde{U}\to U$ whose fiber over $(C, F)\in U$ 
corresponds to the labelings of the two nodes $C\cap F$ of $C+F+\Sigma$. 
The remaining node $F\cap\Sigma$ and the two tangents at each node 
are respectively distinguished by the irreducible decomposition of $C+F+\Sigma$. 
Thus we will obtain a birational lift 
$\widetilde{U}/{\aut}({\F}_4)\dashrightarrow{\cove}_{8,1}$ of $\mathcal{P}$. 
Since ${\Or}(A_L)=\{\pm1\}$ for the invariant lattice $L=U\oplus E_6$, 
we actually have ${\cove}_{8,1}=\mathcal{M}_{8,1}$. 
We also have $\widetilde{U}/{\aut}({\F}_4)=U/{\aut}({\F}_4)$ 
because the hyperelliptic involutions \eqref{eqn: HE invol} of $C$  
give the covering transformation of $\widetilde{U}\to U$. 
\end{proof}

\begin{proposition}\label{rational (8,1)}
The quotient $U/{\aut}({\F}_4)$ is rational. 
Therefore $\mathcal{M}_{8,1}$ is rational. 
\end{proposition}

\begin{proof}
This is a consequence of the slice method for the projection 
$|L_{2,0}|\times|L_{0,1}|\to|L_{0,1}|$, 
Lemma \ref{linear system} $(1)$, 
and Miyata's theorem \cite{Miy}. 
\end{proof}

Via the fixed curve map, $\mathcal{M}_{8,1}$ becomes birationally a fibration over $\mathcal{H}_3$ 
whose general fibers are the hyperelliptic pencils. 
The latter can also be identified with the moduli of pointed hyperelliptic curves of genus $3$. 
The degeneration relation to the fixed curve map for $\mathcal{M}_{6,2}$ is visible 
by regarding the hyperelliptic pencils as natural conics in the canonical systems.

\subsection{The rationality of $\mathcal{M}_{10,0}$}

We consider curves on ${\F}_4$. 
Let $U\subset|L_{2,0}|\times|L_{0,1}|$ be the locus of those $(C, F)$ such that 
$C$ is smooth and tangent to $F$. 
Taking the triple covers of ${\F}_4$ branched over $C+F+\Sigma$,  
we obtain a period map 
$\mathcal{P}\colon U/{\aut}({\F}_4)\dashrightarrow\mathcal{M}_{10,0}$. 

\begin{proposition}\label{period map (10,0)}
The map $\mathcal{P}$ is birational. 
\end{proposition}

\begin{proof}
The singularities of $C+F+\Sigma$ are the node $F\cap\Sigma$ and the tacnode $F\cap C$, 
which are obviously distinguished. 
Also the two branches at each double point are distinguished 
by the irreducible decomposition of $C+F+\Sigma$. 
Thus we need no additional marking to obtain a birational lift 
$U/{\aut}({\F}_4)\dashrightarrow{\cove}_{10,0}$ of $\mathcal{P}$.  
Since the invariant lattice $L\simeq U\oplus E_8$ is unimodular, we have 
${\cove}_{10,0}=\mathcal{M}_{10,0}$. 
\end{proof}

\begin{proposition}\label{rational (10,0)}
The quotient $U/{\aut}({\F}_4)$ is rational. 
Therefore $\mathcal{M}_{10,0}$ is rational. 
\end{proposition}

\begin{proof}
We have the ${\aut}({\F}_4)$-equivariant morphism $\psi\colon U\to{\F}_4$, $(C, F)\mapsto C\cap F$. 
The $\psi$-fibers are open sets of sub-linear systems of $|L_{2,0}|$. 
Then our assertion follows from the slice method for $\psi$, 
Lemma \ref{stab of (pt, *)}, and Miyata's theorem. 
\end{proof}

By Proposition \ref{period map (10,0)}, 
$\mathcal{M}_{10,0}$ is birational to the divisor of Weierstrass points 
in the moduli of pointed genus $3$ hyperelliptic curves, via the fixed curve map.

%%%%%%%%%%%%%%
% Section : g=2
%%%%%%%%%%%%%%

\section{The case $g=2$}\label{sec:g=2}

\subsection{The rationality of $\mathcal{M}_{6,4}$}

We consider curves on ${\F}_3$. 
Let $U\subset|L_{2,0}|\times|L_{1,0}|$ be the open set of pairs $(C, H)$ such that 
$C$ and $H$ are smooth and transverse to each other. 
For $(C, H)\in U$ we associate the ${\Q}$-divisor $C+\frac{1}{2}(H+\Sigma)$ as a mixed branch. 
The associated Eisenstein $K3$ surface has invariant $(g, k)=(2, 0)$, 
and we obtain a period map 
$\mathcal{P}\colon U/{\aut}({\F}_3)\dashrightarrow\mathcal{M}_{6,4}$. 

\begin{proposition}\label{period map (6,4)}
The map $\mathcal{P}$ is birational. 
\end{proposition}

\begin{proof}
We argue as in Example \ref{ex3}. 
Let $\widetilde{U}\subset U\times({\F}_3)^6$ be the locus of 
$(C, H, p_1,\cdots, p_6)$ such that $C\cap H=\{ p_i\}_{i=1}^6$. 
The space $\widetilde{U}$ is an $\frak{S}_6$-cover of $U$ endowing 
$C+H+\Sigma$ with labelings of its six nodes. 
For $(C, \cdots, p_6)\in\widetilde{U}$ 
we make the following birational transformations successively: 
$(1)$ blow-up $p_1+p_2+p_3+p_4$; 
$(2)$ blow-down the strict transforms of the $\pi$-fibers through $p_3+p_4$; and 
$(3)$ blow-down the strict transforms of $H+\Sigma$. 
Then $C$ is transformed to a bidegree $(3, 3)$ curve $C^{\dag}$ 
on $Q={\proj}^1\times{\proj}^1$ having two nodes, say $q_1$ and $q_2$, 
which are respectively the blown-down points of $H$ and $\Sigma$. 
The $(-1)$-curves over $p_1$ and $p_2$ turn to complementary ruling fibers of $Q$, 
the $\pi$-fibers through $p_3$ and $p_4$ turn to the tangents of $C^{\dag}$ at $q_2$, 
and the points $p_5$ and $p_6$ turn to the tangents of $C^{\dag}$ at $q_1$. 
Thus $C^{\dag}$ is naturally endowed with a labeling of the nodes and tangents at them, 
and the two rulings of $Q$ are also distinguished (by $p_1$ and $p_2$).  
Remembering such labellings, one may reverse this construction. 
Therefore, if we denote by $\widetilde{V}$ the space of 
two-nodal curves of bidegree $(3, 3)$ on $Q$ endowed with 
suitable labelings of the nodes and tangents there, %(we have $2^3$ choices of labelings) 
we have a natural birational equivalence 
$\widetilde{U}/{\aut}({\F}_3)\sim\widetilde{V}/({\PGL}_2)^2$. 
%Here $({\PGL}_2)^2$ is the identity component of ${\aut}(Q)$ preserving the two rulings. 
Using the recipe in \S \ref{ssec: recipe}, 
we then see that $\mathcal{P}$ lifts to a birational map 
$\widetilde{U}/{\aut}({\F}_3)\dashrightarrow{\cove}_{6,4}$. 
Since ${\aut}({\F}_3)$ acts on $U$ almost freely, 
$\widetilde{U}/{\aut}({\F}_3)$ is an $\frak{S}_6$-cover of $U/{\aut}({\F}_3)$. 
On the other hand, we have 
$|{\rm O}(A_L)|=|{\rm GO}^-(4, 3)|=2\cdot6!$ 
for the invariant lattice $L=U(3)\oplus A_2^2$. 
Hence the projection ${\cove}_{6,4}\to\mathcal{M}_{6,4}$ also has degree $6!$. 
\end{proof}

\begin{proposition}\label{rational (6,4)}
The quotient $U/{\aut}({\F}_3)$ is rational. 
Therefore $\mathcal{M}_{6,4}$ is rational. 
\end{proposition}

\begin{proof}
We consider the ${\aut}({\F}_3)$-equivariant map 
\begin{equation*}
\psi : U \to |L_{1,0}|\times|L_{1,0}|, \qquad (C, H)\mapsto (H', H), 
\end{equation*}
where $H'=\varphi(C)$ is as defined in \eqref{average}. 
By Lemma \ref{linear system} $(2)$,  
the group ${\aut}({\F}_3)$ acts on $|L_{1,0}|\times|L_{1,0}|$ almost transitively,  
and the stabilizer $G$ of a general point $(H', H)$ is the permutation group of the three points $H\cap H'$. 
The fiber $\psi^{-1}(H', H)$ is an open set of a linear system ${\proj}V\subset|L_{2,0}|$ as before, 
with $G$ acting on $V$ linearly. 
Hence we have $U/{\aut}({\F}_3)\sim {\proj}V/G$ by the slice method.  
It is well-known that ${\proj}V'/\frak{S}_3$ is rational for any $\frak{S}_3$-representation $V'$. 
(Apply the no-name lemma \cite{Do} to the irreducible decomposition of $V'$.) 
\end{proof}

The restriction of $|L_{1,0}|$ to a smooth $L_{2,0}$-curve $C$ gives $|3K_C|$. 
Thus the fixed curve map makes $\mathcal{M}_{6,4}$ birationally a fibration over $\mathcal{M}_2$ 
whose general fibers are the quotients of the tri-canonical systems by the hyperelliptic involutions.

\subsection{The rationality of $\mathcal{M}_{8,3}$}

We consider curves on ${\F}_6$. 
Let $U\subset|L_{2,0}|$ be the locus of irreducible three-nodal curves. 
Associating to $C\in U$ the triple cover of ${\F}_6$ branched over $C+\Sigma$,  
we obtain a period map $U/{\aut}({\F}_6)\dashrightarrow\mathcal{M}_{8,3}$. 
In Example \ref{ex1} we proved that this map is birational. 

\begin{proposition}\label{rational (8,3)}
The quotient $U/{\aut}({\F}_6)$ is rational. 
Therefore $\mathcal{M}_{8,3}$ is rational. 
\end{proposition}

\begin{proof}
We argue as in the proof of Proposition \ref{rational (6,2)}. 
First we have a birational equivalence 
\begin{equation}\label{ele trans (8,3)}
U/{\aut}({\F}_6) \sim (|L_{2,0}|\times|L_{0,3}|)/{\aut}({\F}_3) 
\end{equation}
via the elementary transformations at the nodes of $C\in U$. 

Next we apply the slice method to the ${\aut}({\F}_3)$-equivariant map 
\begin{equation*}
\psi = (\varphi, {\rm id}) : |L_{2,0}|\times|L_{0,3}| \dashrightarrow |L_{1,0}|\times|L_{0,3}|,  
\end{equation*}
where $\varphi$ is as defined in \eqref{average}. 
By Lemma \ref{linear system} (2), ${\aut}({\F}_3)$ acts on $|L_{1,0}|\times|L_{0,3}|$ almost transitively. 
If we normalize $H\in|L_{1,0}|$ to be $H_0$ in \S \ref{ssec: Hirze}, 
the stabilizer $G$ of $(H_0, \sum_iF_i)\in|L_{1,0}|\times|L_{0,3}|$ with $\sum_iF_i$ general is given by 
\begin{equation*}
1 \to  \{ g_{\alpha,0}\}_{\alpha\in{\C}^{\times}} \to G \to \frak{S}_3 \to 1, 
\end{equation*}
where $g_{\alpha,0}$ is as defined in \eqref{$R$-action in coordinate}, 
and $\frak{S}_3$ is the stabilizer in ${\aut}(\Sigma)$ of the three points $\sum_iF_i|_{\Sigma}$. 
On the other hand, we identify $H^0(L_{2,0})$ with the linear space 
$\{ \sum_{i=0}^2 f_i(x_3)y_3^{2-i}\}$ as in \eqref{def eq}. 
Then the fiber $\psi^{-1}(H_0, \sum_iF_i)$ is an open set of the linear subspace ${\proj}V\subset|L_{2,0}|$ defined by $f_1\equiv0$. 
Therefore we have 
\begin{equation*}
(|L_{2,0}|\times|L_{0,3}|)/{\aut}({\F}_3) \sim {\proj}V/G. 
\end{equation*}

The elements $g_{\alpha,0}\in G$ act on $V$ by the same equation as \eqref{action g_alpha}. 
Thus, if we consider the hyperplane $W=\{ f_0=0\}$ of $V$, 
we have the $G$-decomposition $V={\C}y_3^2\oplus W$, and hence ${\proj}V/G\sim{\proj}W/\frak{S}_3$. 
Since $\frak{S}_3$ acts on $W$ linearly, ${\proj}W/\frak{S}_3$ is rational as is well-known. 
\end{proof}

By \eqref{ele trans (8,3)}, the general fibers of the fixed curve map $\mathcal{M}_{8,3}\to\mathcal{M}_2$ are 
birationally identified with the third symmetric products of the hyperelliptic pencils. 

\subsection{The rationality of $\mathcal{M}_{10,2}$}

We consider curves on ${\F}_4$. 
Let $U\subset|L_{2,0}|\times|L_{0,1}|$ be the locus of pairs $(C, F)$ such that 
$C$ is irreducible and one-nodal, and $F$ is transverse to $C$. 
Considering the $-\frac{3}{2}K_{{\F}_4}$-curves $C+F+\Sigma$,  
we obtain a period map $\mathcal{P}\colon U/{\aut}({\F}_4)\dashrightarrow\mathcal{M}_{10,2}$. 

\begin{proposition}\label{period map (10,2)}
The map $\mathcal{P}$ is birational. 
\end{proposition}

\begin{proof}
We label the two tangents of $C$ at the node and the two points $C\cap F$ independently: 
this is realized by an $\frak{S}_2\times\frak{S}_2$-cover $\widetilde{U}\to U$. 
The two tangents at each point of $F\cap (C+\Sigma)$ are distinguished by the irreducible decomposition of $C+F+\Sigma$. 
Therefore we have a birational lift $\widetilde{U}/{\aut}({\F}_4)\dashrightarrow{\cove}_{10,2}$ 
of $\mathcal{P}$ as before. 
Since the invariant lattice $L$ is isometric to $U\oplus E_6\oplus A_2$, 
we have ${\rm O}(A_L)\simeq({\Z}/2{\Z})^2$ so that 
${\cove}_{10,2}$ is a double cover of $\mathcal{M}_{10,2}$. 
On the other hand, the hyperelliptic involution $\iota_C$ defined in \eqref{eqn: HE invol} exchanges 
the two tangents of $C$ and the two points $C\cap F$ simultaneously. 
Therefore $\widetilde{U}/{\aut}({\F}_4)\dashrightarrow U/{\aut}({\F}_4)$ is also a double covering. 
\end{proof}

\begin{proposition}\label{rational (10,2)}
The quotient $U/{\aut}({\F}_4)$ is rational. 
Hence $\mathcal{M}_{10,2}$ is rational. 
\end{proposition}

\begin{proof}
We apply the slice method to the ${\aut}({\F}_4)$-equivariant map 
\begin{equation*}
U\to{\F}_4\times|L_{0,1}|, \qquad (C, F)\mapsto({\rm Sing}(C), F), 
\end{equation*}
whose general fiber is an open set of a sub-linear system of $|L_{2,0}|$.  
Then we may use Lemma \ref{stab of (pt, *)} and Miyata's theorem. 
\end{proof}

Let $\mathcal{X}_2$ be the moduli space of pointed genus $2$ curves (whose general fibers over $\mathcal{M}_2$ are the hyperelliptic pencils). 
As before, we see that the fixed curve map makes $\mathcal{M}_{10,2}$ birational to the fibration 
$\mathcal{X}_2\times_{\mathcal{M}_2}\mathcal{X}_2$ over $\mathcal{M}_2$.

\subsection{The rationality of $\mathcal{M}_{12,1}$}

We consider curves on ${\F}_4$. 
Let $U\subset|L_{2,0}|\times|L_{0,1}|$ be the locus of those $(C, F)$ such that 
$C$ is irreducible and one-nodal, and $F$ is tangent to $C$ at a smooth point. 
By considering the triple covers of ${\F}_4$ branched over $C+F+\Sigma$,  
we obtain a period map 
$\mathcal{P}\colon U/{\aut}({\F}_4)\dashrightarrow\mathcal{M}_{12,1}$. 

\begin{proposition}\label{period map (12,1)}
The map $\mathcal{P}$ is birational. 
\end{proposition}

\begin{proof}
As before, we consider a double cover $\widetilde{U}\to U$ whose fiber over $(C, F)\in U$ 
corresponds to the labelings of the two branches of $C$ at the node. 
The rest singularities of $C+F+\Sigma$ are the node $F\cap\Sigma$ and the tacnode $F\cap C$, 
where the branches of $C+F+\Sigma$ are distinguished by the irreducible decomposition of $C+F+\Sigma$. 
Following the recipe in \S \ref{ssec: recipe}, 
we will obtain a birational lift $\widetilde{U}/{\aut}({\F}_4)\dashrightarrow{\cove}_{12,1}$ of $\mathcal{P}$.
Since the invariant lattice $L$ is isometric to $U\oplus E_8\oplus A_2$, 
we have ${\rm O}(A_L)\simeq\{\pm1\}$ so that ${\cove}_{12,1}=\mathcal{M}_{12,1}$. 
We also have $\widetilde{U}/{\aut}({\F}_4)=U/{\aut}({\F}_4)$ 
because the hyperelliptic involutions \eqref{eqn: HE invol} give the covering transformation of $\widetilde{U}\to U$. 
\end{proof}

\begin{proposition}\label{rational (12,1)}
The quotient $U/{\aut}({\F}_4)$ is rational. 
Hence $\mathcal{M}_{12,1}$ is rational. 
\end{proposition}

\begin{proof}
Consider the ${\aut}({\F}_4)$-equivariant map 
\begin{equation*}
\psi : U\to{\F}_4\times{\F}_4, \qquad (C, F)\mapsto({\rm Sing}(C), C\cap F). 
\end{equation*}
The $\psi$-fiber over a general $(p, q)$ is an open set of the linear system in $|L_{2,0}|$ of 
curves singular at $p$ and branched at $q$ over $\Sigma$.  
Then we apply the slice method for $\psi$, and use Lemma \ref{stab of (pt, *)} and Miyata's theorem. 
\end{proof}

Let $\mathcal{W}\subset\mathcal{X}_2$ be the divisor of Weierstrass points. 
Then the fixed curve map identifies $\mathcal{M}_{12,1}$ birationally with the fibration 
$\mathcal{X}_2\times_{\mathcal{M}_2}\mathcal{W}$ over $\mathcal{M}_2$.

%%%%%%%%%%%%%%
% Section : g=1
%%%%%%%%%%%%%%

\section{The case $g=1$}\label{sec:g=1}

In this section we study the case $g=1$. 
The cases $k=0, 1$ are beyond the previous method and we have to analyze symmetry by the Weyl groups $W(E_6)$, $W(F_4)$ respectively. 
When $k\geq4$, we have ${\dim}{\moduli}\leq2$ so that it is enough to give a unirational parameter space that dominates ${\moduli}$. 
But for future reference, we shall take extra effort to present degree $1$ period maps.

\subsection{The rationality of $\mathcal{M}_{8,5}$}\label{ssec:(8,5)}

Let us first recall few basic facts about cubic surfaces. 
Let $Y\subset{\proj}^3$ be a smooth cubic surface. 
For each point $p\in Y$, the tangent plane section of $Y$ at $p$ gives the unique $-K_Y$-curve $C_p$ singular at $p$. 
When $C_p$ is irreducible, it is cuspidal at $p$ if and only if 
$p$ lies on the intersection of $Y$ with its Hessian quartic; otherwise $C_p$ is nodal at $p$. 

A \textit{marking} of $Y$ is an isometry $I_{1,6}=\langle1\rangle\oplus\langle-1\rangle^6 \to NS_Y$ of lattices 
which maps $3h-\sum_{i=1}^{6}e_i$ to $-K_Y$, 
where $h, e_1,\cdots, e_6$ is a natural orthogonal basis of $I_{1,6}$. 
Such a marking realizes $Y$ as the blow-up of ${\proj}^2$ at six general points $p_1,\cdots, p_6$,  
for which the pullback of ${\Oplane}(1)$ corresponds to $h$ and 
the $(-1)$-curve over $p_i$ corresponds to $e_i$.  
By that blow-down $Y\to{\proj}^2$, the $-K_Y$-curves are mapped to plane cubics through $p_1,\cdots, p_6$. 
The stabilizer in ${\Or}(I_{1,6})$ of the vector $3h-\sum_ie_i$ is the Weyl group $W(E_6)$. 
It acts transitively on the set of markings of $Y$. 
Equivalently, $W(E_6)$ transforms the ordered point set $(p_1,\cdots, p_6)$ to another one up to ${\PGL}_3$. 
To sum up, the moduli space ${\Mmcub}$ of marked cubic surfaces is identified with the configuration space of six general points in ${\proj}^2$, 
on which $W(E_6)$ acts with the quotient the moduli space ${\Mcub}$ of smooth cubic surfaces.

Now we consider the parameter space 
$U\subset|{\Ospace}(3)|\times{\proj}^3\times|{\Ospace}(1)|$ 
of triplets $(Y, p, H)$ such that 
$({\rm i})$ $Y$ is a smooth cubic surface, 
$({\rm ii})$ $p\in Y$, 
$({\rm iii})$ the $-K_Y$-curve $C_p$ is irreducible and cuspidal at $p$, and 
$({\rm iv})$ the $-K_Y$-curve $C=H|_Y$ is smooth and tangent to $C_p$ at $p$. 
Note that $C$ and $C_p$ do not intersect outside $p$. 
To such a triplet $(Y, p, H)$ we associate the mixed branch $C+\frac{1}{2}C_p$ on $Y$. 
(Strictly speaking, this does not satisfy the conditions on singularity of mixed branch. 
But we can resolve $C+\frac{1}{2}C_p$ following the process \eqref{resol process} to pass to a smooth mixed branch. 
Thus we shall abuse the terminology.) 
By associating Eisenstein $K3$ surfaces as explained in \S \ref{ssec:mixed branch}, 
we obtain a period map 
$\mathcal{P}\colon U/{\PGL}_4 \dashrightarrow \mathcal{M}_{8,5}$.

\begin{proposition}\label{birat (8,5)}
The period map $\mathcal{P}$ is birational. 
\end{proposition}

\begin{proof}
To calculate the degree of $\mathcal{P}$, we make use of markings of $Y$ in an auxiliary way (cf.~\cite{Ma2} \S 12). 
Let $\mu$ be a marking of $Y$ and $\pi\colon Y\to{\proj}^2$ the corresponding blow-down. 
The  pair $(C, C_p)$ of $-K_Y$-curves is mapped to the pair $(B_1, B_2)=(\pi(C), \pi(C_p))$ of irreducible plane cubics such that 
$({\rm i})$ $B_2$ is cuspidal and  
$({\rm ii})$ $B_1$ is smooth, tangent to $B_2$ at its cusp, and transverse to $B_2$ elsewhere. 
The six intersection points $B_1\cap B_2\backslash{\rm Sing}(B_2)$ are the blown-up points of $\pi$ 
and hence ordered by $\mu$.  
This leads us to consider the space 
$\widetilde{U}\subset|{\Oplane}(3)|^2\times({\proj}^2)^6$ 
of those $(B_1, B_2, p_1,\cdots, p_6)$ such that 
the cubics $B_1, B_2$ satisfy the conditions (i), (ii) above and 
that $B_1\cap B_2\backslash{\rm Sing}(B_2)=\{ p_1,\cdots, p_6\}$. 
Regarding ${\Mmcub}$ as the configuration space of six points in ${\proj}^2$, 
we may identify $\widetilde{U}/{\PGL}_3$ birationally with the moduli space of 
marked cubic surfaces $(Y, \mu)$ with mixed branches $C+\frac{1}{2}C_p$ such that $(Y, p, C)\in U$. 
Thus we have a quotient map $\widetilde{U}/{\PGL}_3 \to U/{\PGL}_4$ by $W(E_6)$, 
where $W(E_6)$ acts on $\widetilde{U}/{\PGL}_3$ by the Cremona transformations. 

The point is that the period map $\mathcal{P}$ lifts to a birational map 
${\lift}\colon \widetilde{U}/{\PGL}_3 \dashrightarrow {\cove}_{8,5}$. 
Indeed, we may view $\widetilde{U}$ as parametrizing mixed branches $B_1+\frac{1}{2}B_2$ on ${\proj}^2$ 
endowed with labelings of the six points $B_1\cap B_2\backslash{\rm Sing}(B_2)$. 
The composition 
\begin{equation}\label{eqn: W(E_6)-symmetry}
\widetilde{U}/{\PGL}_3 \to U/{\PGL}_4 \stackrel{\mathcal{P}}{\to} \mathcal{M}_{8,5}
\end{equation}
associates Eisenstein $K3$ surfaces to those labelled mixed branches in the way of \S \ref{ssec:mixed branch}. 
Then we can follow the idea in Remark \ref{variant recipe}. 
The ordering of the six points induces a marking of $L(X, G)$; 
conversely, from this marking we can recover the labelled mixed branch 
by looking the $(-2)$-curves over the six points and the pullback of ${\Oplane}(1)$ to $X$ or $\hat{X}$. 
This enables us to construct a lift\footnote{
Since the surjectivity of ${\rm U}(E)\to{\Or}(A_E)$ for the Eisenstein lattice $E=U^2\oplus A_2^5$ is yet uncertain at this moment, 
here we should narrow the moduli interpretation of ${\cove}_{8,5}$ as indicated in the footnote in p.\pageref{footnote1}. 
The lattice-markings induced from our labelled mixed branches do meet the requirement there, because, e.g., 
the connectivity of $\widetilde{U}$ ensures that the Eisenstein $K3$ surfaces can be deformed to each other preserving the markings.  
} 
of \eqref{eqn: W(E_6)-symmetry} to ${\cove}_{8,5}$ 
and show that it has degree $1$.

The Galois group of ${\cove}_{8,5}\to\mathcal{M}_{8,5}$ is a subgroup of ${\rm O}(A_E)/\pm1$. %for the invariant lattice $L=U(3)\oplus A_2^3$. 
By \cite{Atlas} we have $|{\rm O}(A_E)|=|{\rm GO}(5, 3)|=2\cdot|W(E_6)|$. 
Comparing the two coverings 
${\cove}_{8,5}\to\mathcal{M}_{8,5}$ and $\widetilde{U}/{\PGL}_3 \to U/{\PGL}_4$, 
we conclude that the Galois group is actually whole ${\rm O}(A_E)/\pm1$ and that $\mathcal{P}$ has degree $1$. 
\end{proof}

\begin{proposition}\label{rational (8,5)}
The quotient $U/{\PGL}_4$ is rational. 
Therefore $\mathcal{M}_{8,5}$ is rational. 
\end{proposition}

\begin{proof}
Let $V\subset|{\Ospace}(3)|\times{\proj}^3$ be the locus of pairs $(Y, p)$ such that 
$p$ lies on the intersection of $Y$ with its Hessian quartic. 
We have a natural projection $U\to V$, whose fibre over $(Y, p)$ is an open set of 
the pencil in $|{\Ospace}(1)|$ of planes that contain the tangent line of $C_p$ at $p$. 
Therefore $U$ is birationally the projectivization of an ${\SL}_4$-linearized vector bundle $\mathcal{E}$ over $V$. 
The element $\sqrt{-1}\in{\SL}_4$ acts on $\mathcal{E}$ by the scalar multiplication by $\sqrt{-1}$. 
We tensor $\mathcal{E}$ with the pullback $\mathcal{L}$ of the hyperplane bundle on $|{\Ospace}(3)|$, 
on which $\sqrt{-1}\in{\SL}_4$ acts by the multiplication by $-\sqrt{-1}$. 
Then $\mathcal{E}\otimes\mathcal{L}$ is ${\PGL}_4$-linearized, and 
${\proj}(\mathcal{E}\otimes\mathcal{L})$ is canonically identified with ${\proj}\mathcal{E}$. 
Since ${\PGL}_4$ acts on $V$ almost freely, we may use the no-name lemma for  
$\mathcal{E}\otimes\mathcal{L}$ to obtain 
\begin{equation*}
U/{\PGL}_4 \sim {\proj}(\mathcal{E}\otimes\mathcal{L})/{\PGL}_4 \sim {\proj}^1\times(V/{\PGL}_4). 
\end{equation*}

Next let $W$ be the space of flags $p\in l\subset P\subset{\proj}^3$, 
where $l$ is a line and $P$ is a plane. 
We have the ${\PGL}_4$-equivariant map 
\begin{equation*}\label{eqn: flag}
\varphi : V\to W, \qquad (Y, p)\mapsto(p, T_pC_p, T_pY), 
\end{equation*}
whose fiber is a linear subspace of $|{\Ospace}(3)|$. 
The group ${\SL}_4$ acts on $W$ transitively with a connected and solvable stabilizer. 
Therefore we may apply the slice method to $\varphi$ and then use Miyata's theorem to see that $V/{\PGL}_4$ is rational. 
\end{proof}

We can also use $C+C_p$ as $-2K_Y$-curves to obtain 2-elementary $K3$ surfaces with 
$(r, a, \delta)=(14, 6, 0)$ (cf.~ \cite{Ma2}). 
This turns out to be a canonical construction for 
general members of their moduli space $\mathcal{M}_{14,6,0}$. 
Thus we have a geometric birational map 
$\mathcal{M}_{8,5}\dashrightarrow\mathcal{M}_{14,6,0}$ via $U/{\PGL}_4$. 
Since $\mathcal{M}_{14,6,0}$ is proven to be rational in \cite{Ma2} by another method, this offers a second proof of the rationality of $\mathcal{M}_{8,5}$.

\subsection{The rationality of $\mathcal{M}_{10,4}$}\label{ssec:(10,4)}

We study $\mathcal{M}_{10,4}$ using cubic surfaces with Eckardt points. 
In addition to the anti-canonical model and the blown-up ${\proj}^2$ model as used in \S \ref{ssec:(8,5)},  
we will also use the Sylvester form of (general) smooth cubic surfaces $Y$:
\begin{equation}\label{Sylvester}
\sum_{i=0}^{4} \lambda_iX_i^3 = \sum_{i=0}^{4} X_i = 0, \qquad \lambda_i\in{\C}, 
\end{equation}
where $[X_0,\cdots,X_4]$ is the homogeneous coordinate of ${\proj}^4$. 
This expression of $Y$ is unique up to the permutations of $\lambda_0,\cdots, \lambda_4$ and the scalar multiplications on $(\lambda_0,\cdots, \lambda_4)$. 
For details about Eckardt points, we refer to \cite{Se}, \cite{Na} and \cite{D-G-K}.

Let $Y\subset{\proj}^3$ be a smooth cubic surface. 
A point $p\in Y$ is called an \textit{Eckardt point} if the tangent plane section $C_p=T_pY|_Y$ is a union of three lines meeting at $p$. 
In the Sylvester form \eqref{Sylvester}, $Y$ has such a point if and only if $\prod_{i<j}(\lambda_i-\lambda_j)=0$. 
For simplicity, we may assume $\lambda_3=\lambda_4$. 
Then $p=[0,0,0,1,-1]$ is an Eckardt point of $Y$. 
The surface $Y$ has an involution $\iota$, called \textit{harmonic homology}, given by $X_3\leftrightarrow X_4$ and $X_i\mapsto X_i$ for $i\leq2$. 
If $Y$ is general in the locus $\{ \lambda_3=\lambda_4 \}$, it has no other nontrivial automorphism.  

\begin{lemma}\label{harmonic homology}
The harmonic homology $\iota$ acts trivially on the linear space of anti-canonical forms vanishing at $p$. 
\end{lemma}

\begin{proof}
Let $H=\sum_{i=0}^{4}X_i$. 
Since $Y\subset\{ H=0\}$ is anti-canonically embedded, we may identify $H^0(-K_Y)$ with $H^0(\mathcal{O}_{{\proj}^4}(1))/{\C}H$. 
If we express linear forms on ${\proj}^4$ as $\sum_i\alpha_iX_i$, the space in question is identified with the hyperplane $\{ \alpha_3=\alpha_4\}\subset H^0(-K_Y)$. 
Then the assertion holds apparently. 
\end{proof}

Now we consider the locus $U\subset|{\Ospace}(3)|\times{\proj}^3\times|{\Ospace}(1)|$ of triplets $(Y, p, H)$ such that 
(i) $Y$ is smooth, (ii) $p$ is an Eckardt point of $Y$, and (iii) the $-K_Y$-curve $C=H|_Y$ is smooth and passes through $p$. 
By using $C+\frac{1}{2}C_p$ as mixed branches, we obtain Eisenstein $K3$ surfaces with $(g, k)=(1, 1)$. 
We thus have a period map $\mathcal{P}\colon U/{\PGL}_4\to\mathcal{M}_{10,4}$. 

In order to show that $\mathcal{P}$ is birational, we describe $U/{\PGL}_4$ in a different way. 
Let ${\Mcub}$, ${\Mmcub}$ be the moduli spaces defined in \S \ref{ssec:(8,5)}, 
and $\pi\colon{\Mmcub}\to{\Mcub}$ be the quotient map by the Weyl group $W(E_6)$. 
We have a universal family $f\colon\mathcal{Y}\to{\Mmcub}$ of marked cubic surfaces, on which $W(E_6)$ acts equivariantly (cf. \cite{Na} \S 1, \cite{Ma2} \S 12.1). 
Let $\mathcal{E}\subset{\Mcub}$ be the codimension $1$ locus of cubic surfaces having exactly one Eckardt point. 
Then $\pi^{-1}(\mathcal{E})$ has $45$ irreducible components which are permuted transitively by $W(E_6)$. 
Let $\widetilde{\mathcal{E}}\subset\pi^{-1}(\mathcal{E})$ be either one component and $G\subset W(E_6)$ the stabilizer of $\widetilde{\mathcal{E}}$. 
($G$ is the Weyl group $W(F_4)$.) 
The center of $G$ is ${\Z}/2{\Z}$, which acts on $\widetilde{\mathcal{E}}$ trivially and on the restricted family 
\begin{equation*}
f' = f|_{f^{-1}(\widetilde{\mathcal{E}})} : f^{-1}(\widetilde{\mathcal{E}}) \to \widetilde{\mathcal{E}} 
\end{equation*} 
by the harmonic homologies. 
We consider the sub-vector bundle $\mathcal{F}\subset f'_{\ast}K_{f'}^{-1}$ whose fibers are the linear spaces of anti-canonical forms vanishing at the Eckardt points. 
Note that $\mathcal{F}$ is $G$-linearized because $f_{\ast}K_{f}^{-1}$ is $W(E_6)$-linearized. 
Forgetting the markings of cubic surfaces, we see that $U/{\PGL}_4$ is birationally identified with ${\proj}\mathcal{F}/G$. 
Now we can prove  

\begin{proposition}\label{birat (10,4)}
The period map $\mathcal{P}\colon {\proj}\mathcal{F}/G\to\mathcal{M}_{10,4}$ is birational. 
\end{proposition}

\begin{proof}
We show that $\mathcal{P}$ lifts to a birational map ${\proj}\mathcal{F}\to{\cove}_{10,4}$. 
Let $V\subset({\proj}^2)^6$ be the locus of six distinct points $(p_1,\cdots,p_6)$ such that 
the three lines $L_i=\overline{p_ip_{i+3}}$ ($1\leq i\leq3$) intersect at one point, say $p$. 
Regarding ${\Mmcub}$ as the configuration space of six points in ${\proj}^2$, 
we have a natural birational identification $V/{\PGL}_3\sim \widetilde{\mathcal{E}}$. 
Therefore, if $\widetilde{U}\subset V\times|{\Oplane}(3)|$ is the locus of those $(p_1,\cdots,p_6, C)$ such that 
$C$ is smooth and passes through $p_1,\cdots,p_6, p$, then ${\proj}\mathcal{F}$ is birationally identified with $\widetilde{U}/{\PGL}_3$. 
We may regard $\widetilde{U}$ as parametrizing mixed branches $C+\frac{1}{2}\sum_iL_i$ endowed with labelings of the six intersection points 
$C\cap\sum_iL_i\backslash p$ that are compatible with the irreducible decomposition of $\sum_iL_i$. 
Then the composition 
\begin{equation*}
\widetilde{U}/{\PGL}_3 \sim {\proj}\mathcal{F} \to {\proj}\mathcal{F}/G \stackrel{\mathcal{P}}{\to} \mathcal{M}_{10,4}
\end{equation*}
maps such a labelled mixed branch $C+\frac{1}{2}\sum_iL_i$ to the Eisenstein $K3$ surface associated as in \S \ref{ssec:mixed branch}. 
Hence by arguing as in the proof of Proposition \ref{birat (8,5)}, 
we will obtain a desired birational lift $\widetilde{U}/{\PGL}_3\to{\cove}_{10,4}$. 

The degree of ${\cove}_{10,4}\to\mathcal{M}_{10,4}$ divides $|{\Or}(A_E)|/2=|{\rm GO}^+(4, 3)|/2=(4!)^2$ 
for the Eisenstein lattice $E=U^2\oplus A_2^4$. 
On the other hand, ${\proj}\mathcal{F}\to{\proj}\mathcal{F}/G$ has degree $|G|/2=|W(E_6)|/90=(4!)^2$ 
because the center of $G$ acts on ${\proj}\mathcal{F}$ trivially by Lemma \ref{harmonic homology}. 
Comparing the two projections ${\cove}_{10,4}\to\mathcal{M}_{10,4}$ and ${\proj}\mathcal{F}\to{\proj}\mathcal{F}/G$, 
we find that $\mathcal{P}$ has degree $1$ and that the Galois group of the former is ${\Or}(A_E)/\pm1$. 
\end{proof} 

\begin{proposition}\label{rational (10,4)}
The quotient ${\proj}\mathcal{F}/G$ is rational. 
Therefore $\mathcal{M}_{10,4}$ is rational. 
\end{proposition}

\begin{proof} 
By Lemma \ref{harmonic homology}, the center of $G$ acts on $\mathcal{F}$ trivially. 
Replacing $G$ by its central quotient and applying the no-name lemma to the $G$-linearized vector bundle $\mathcal{F}\to\widetilde{\mathcal{E}}$, 
we have 
\begin{equation*}
{\proj}\mathcal{F}/G \sim {\proj}^2\times(\widetilde{\mathcal{E}}/G) \sim {\proj}^2\times\mathcal{E}. 
\end{equation*}
By the Sylvester form \eqref{Sylvester}, the Eckardt locus $\mathcal{E}$ is biratinal to ${\proj}W/\frak{S}_3$ 
where $W=\{ \lambda_3=\lambda_4\}\subset{\C}^5$ and $\frak{S}_3$ acts on $W$ by the permutations of $(\lambda_0, \lambda_1, \lambda_2)$. 
Therefore $\mathcal{E}$ is rational. 
\end{proof}

\subsection{The rationality of $\mathcal{M}_{12,3}$}

We consider curves on ${\F}_1$. 
Let $V\subset|L_{2,2}|$ be the locus of curves $C$ 
which have a cusp at $C\cap\Sigma$ and are smooth elsewhere. 
($C$ is the blow-up of a plane quartic with a ramphoid cusp.) 
Let $U\subset V\times|L_{1,0}|$ be the open set of pairs $(C, H)$ such that 
$H$ is smooth and transverse to $C$. 
For $(C, H)\in U$ we consider the mixed branch $C+\frac{1}{2}(H+\Sigma)$. 
The associated Eisenstein $K3$ surface has invariant $(g, k)=(1, 2)$. 
Hence we obtain a period map 
$\mathcal{P}\colon U/{\aut}({\F}_1)\dashrightarrow\mathcal{M}_{12,3}$. 

\begin{proposition}\label{period map (12,3)}
The period map $\mathcal{P}$ is birational. 
\end{proposition}

\begin{proof}
This is analogous to Example \ref{ex3} and Proposition \ref{period map (6,4)}: 
we label the four nodes $C\cap H$ by an $\frak{S}_4$-cover $\widetilde{U}\to U$. 
By blowing-up the "first" and "second" nodes and then 
blowing-down the strict transforms of $H$ and $\Sigma$, 
the curve $C$ is transformed to a bidegree $(3, 3)$ curve $C^{\dag}$ on ${\proj}^1\times{\proj}^1$ 
which has a node and a ramphoid cusp. 
The given labeling of $C\cap H$ induces that of 
the tangents of $C^{\dag}$ at the node, and of the two rulings of ${\proj}^1\times{\proj}^1$. 
Then we see as in Example \ref{ex3} that $\mathcal{P}$ lifts to a birational map 
$\widetilde{U}/{\aut}({\F}_1)\dashrightarrow{\cove}_{12,3}$. 
The group ${\aut}({\F}_1)$ acts on $U$ almost freely, 
so that $\widetilde{U}/{\aut}({\F}_1)$ is an $\frak{S}_4$-cover of $U/{\aut}({\F}_1)$. 
On the other hand, we have 
${\rm O}(A_L)\simeq{\rm GO}(3, 3)$ for the invariant lattice $L=U\oplus E_6\oplus A_2^2$. 
Then $|{\rm O}(A_L)|=2\cdot4!$ by \cite{Atlas}, and hence $\mathcal{P}$ has degree $1$. 
\end{proof}

\begin{proposition}\label{period map (12,3)}
The quotient $U/{\aut}({\F}_1)$ is rational. 
Therefore $\mathcal{M}_{12,3}$ is rational. 
\end{proposition}

\begin{proof}
We first apply the slice method to the ${\aut}({\F}_1)$-equivariant map 
\begin{equation*}
\psi : U \to \Sigma \times |L_{1,0}|, \qquad (C, H)\mapsto({\rm Sing}(C), H). 
\end{equation*}
By Lemma \ref{linear system} $(2)$, ${\aut}({\F}_1)$ acts on $\Sigma \times |L_{1,0}|$ almost transitively. 
If we normalize $H$ to be $H_0$, and ${\rm Sing}(C)$ to be the point $p_0=(0, 0)$ in $U_1$, 
then the stabilizer $G_1$ of $(p_0, H_0)\in\Sigma \times |L_{1,0}|$ is 
\begin{equation*}
G_1 = \{ g_{\alpha,0}\}_{\alpha\in{\C}^{\times}} \times ( \{ h_{\beta}\}_{\alpha\in{\C}^{\times}}\ltimes\{ i_{\lambda}\}_{\lambda\in{\C}}) 
\simeq {\C}^{\times} \times ({\C}^{\times}\ltimes{\C}), 
\end{equation*}
where $g_{\alpha,0}$, $h_{\beta}$, $i_{\lambda}$ are as defined in \eqref{$R$-action in coordinate}--\eqref{auto coord 4}. 
The fiber $\psi^{-1}(p_0, H_0)$ is regarded as a (nonlinear) sublocus of $|L_{2,2}|$. 
Then we have 
$U/{\aut}({\F}_1)\sim \psi^{-1}(p_0, H_0)/G_1$. 

Next we apply the slice method to the $G_1$-equivariant map 
\begin{equation*}
\phi : \psi^{-1}(p_0, H_0)\to{\proj}T_{p_0}{\F}_1,  \quad C\mapsto T_{p_0}C, 
\end{equation*}
where $T_{p_0}C$ denotes the unique tangent of $C$ at $p_0$. 
A general $\phi$-fiber is an open set of a \textit{linear} system ${\proj}V\subset|L_{2,2}|$. 
Since $G_1$ acts on ${\proj}T_{p_0}{\F}_1$ almost transitively, 
we have $\psi^{-1}(p_0, H_0)/G_1 \sim {\proj}V/G_2$ for the stabilizer $G_2\subset G_1$ of a general point of ${\proj}T_{p_0}{\F}_1$. 
If we use $y_1^{-1}x_1$ as the inhomogeneous coordinate of ${\proj}T_{p_0}{\F}_1$, 
then $g_{\alpha,0}$ acts on ${\proj}T_{p_0}{\F}_1$ by $\alpha$, 
$h_{\beta}$ by $\beta$, and 
$i_{\lambda}$ trivially. 
This shows that $G_2$ is isomorphic to ${\C}^{\times}\ltimes{\C}$. 
Hence ${\proj}V/G_2$ is rational by Miyata's theorem. 
\end{proof}

\subsection{The rationality of $\mathcal{M}_{14,2}$}

We consider curves on ${\F}_2$. 
Let $U\subset|L_{2,0}|\times|L_{0,2}|$ be the open set of pairs $(C, F_1+F_2)$ 
such that $C$ and $F_1+F_2$ are smooth and transverse to each other. 
We associate the $-\frac{3}{2}K_{{\F}_2}$-branch $C+F_1+F_2+\Sigma$ 
to obtain a period map 
$U/{\aut}({\F}_2)\dashrightarrow\mathcal{M}_{14,2}$. 
In Example \ref{ex2} we proved that this map is birational. 

\begin{proposition}\label{period map (14,2)}
The quotient $U/{\aut}({\F}_2)$ is rational. 
Hence $\mathcal{M}_{14,2}$ is rational. 
\end{proposition}

\begin{proof}
As in the proof of Proposition \ref{rational (6,2)}, we apply the slice method to the ${\aut}({\F}_2)$-equivariant map 
\begin{equation*}
\psi = (\varphi, {\rm id}) : U \to |L_{1,0}|\times|L_{0,2}|, \quad (C, F_1+F_2)\mapsto(H, F_1+F_2), 
\end{equation*}
where $\varphi$ is as defined in \eqref{average}. 
By Lemma \ref{linear system} (2), ${\aut}({\F}_2)$ acts on $|L_{1,0}|\times|L_{0,2}|$ almost transitively. 
If we normalize $H=H_0$ and $F_i=\{ x_i=0\}$, 
the stabilizer $G$ of $(H_0, F_1+F_2)$ is described by the same equation as \eqref{stabilizer section+two fibers}. 
On the other hand, if we identify $H^0(L_{2,0})$ with the linear space $\{ \sum_{i=0}^2 f_i(x_3)y_3^{2-i}\}$ as in \eqref{def eq}, 
the fiber $\psi^{-1}(H_0, F_1+F_2)$ is an open set of the linear subspace ${\proj}V\subset|L_{2,0}|$ defined by $f_1\equiv0$. 
Therefore we have 
\begin{equation*}
U/{\aut}({\F}_2) \sim {\proj}V/G. 
\end{equation*}

Let $W\subset V$ be the hyperplane $\{ f_0=0\}$. 
As in the proof of Proposition \ref{rational (6,2)}, we see that the $G$-representation $V$ decomposes as $V={\C}y_3^2\oplus W$. 
If we consider the $G$-representation $W'=({\C}y_3^2)^{\vee}\otimes W$, 
then ${\proj}V/G$ is birational to $W'/G$. 
Since $W'/G\sim{\C}^{\times}\times({\proj}W'/G)$ and ${\proj}W'/G$ is $2$-dimensional, 
$W'/G$ is rational. 
\end{proof}

\subsection{The rationality of $\mathcal{M}_{16,1}$}

We consider curves on ${\F}_2$. 
Let $U\subset|L_{2,0}|\times|L_{0,1}|^2$ be the locus of triplets $(C, F_1, F_2)$ 
such that $C$ is smooth, $F_1$ is transverse to $C$, and $F_2$ is tangent to $C$. 
Considering the $-\frac{3}{2}K_{{\F}_2}$-branches $C+F_1+F_2+\Sigma$, 
we have a period map 
$\mathcal{P}\colon U/{\aut}({\F}_2)\dashrightarrow\mathcal{M}_{16,1}$. 

\begin{proposition}\label{period map (16,1)}
The map $\mathcal{P}$ is birational. 
\end{proposition}

\begin{proof}
We consider a double cover $\widetilde{U}\to U$ to label the two points $C\cap F_1$. 
The rest datum for $C+F_1+F_2+\Sigma$ are a priori labelled: 
$F_1$ and $F_2$ are distinguished by their intersection with $C$, 
and the two branches at each (tac)node of $C+F_1+F_2+\Sigma$ 
are distinguished by the irreducible decomposition of $C+F_1+F_2+\Sigma$. 
Thus we will obtain a birational lift 
$\widetilde{U}/{\aut}({\F}_2)\dashrightarrow{\cove}_{16,1}$ of $\mathcal{P}$. 
We have $\widetilde{U}/{\aut}({\F}_2)=U/{\aut}({\F}_2)$ due to the hyperelliptic involutions \eqref{eqn: HE invol} of $C$. 
We also have ${\cove}_{16,1}=\mathcal{M}_{16,1}$ 
because ${\Or}(A_L)=\{ \pm1\}$ for the invariant lattice $L=U\oplus E_6\oplus E_8$. 
\end{proof}

Since $U$ is rational and $\mathcal{M}_{16,1}$ has dimension $2$, we see that  

\begin{proposition}\label{rational (16,1)}
The space $\mathcal{M}_{16,1}$ is rational. 
\end{proposition}

By associating to $(C, F_1, F_2)$ the elliptic curve $(C, F_2\cap C)$ with a point $p\in F_1\cap C$, 
we obtain a birational map from $\mathcal{M}_{16,1}$ to the Kummer modular surface for ${\SL}_2({\Z})$, 
whose projection to the modular curve gives the fixed curve map.

\subsection{The rationality of $\mathcal{M}_{18,0}$}

We consider curves on ${\F}_2$. 
Let $U\subset|L_{2,0}|\times|L_{0,2}|$ be the locus of pairs $(C, F_1+F_2)$ 
such that $C$ is smooth, $F_1\ne F_2$, and both $F_i$ are tangent to $C$. 
We obtain a period map 
$\mathcal{P}\colon U/{\aut}({\F}_2)\dashrightarrow\mathcal{M}_{18,0}$ 
by considering the $-\frac{3}{2}K_{{\F}_2}$-branches $C+F_1+F_2+\Sigma$. 

\begin{proposition}\label{period map (18,0)}
The map $\mathcal{P}$ is birational. 
\end{proposition}

\begin{proof}
As before, we distinguish $F_1$ and $F_2$ by a double cover $\widetilde{U}\to U$ to obtain a birational lift 
$\widetilde{U}/{\aut}({\F}_2)\dashrightarrow{\cove}_{18,0}$ of $\mathcal{P}$. 
Since the invariant lattice $L=U\oplus E_8^2$ is unimodular, ${\cove}_{18,0}$ coincides to $\mathcal{M}_{18,0}$. 
On the other hand, for each $(C, F_1+F_2)\in U$, we have an automorphism of ${\F}_2$ preserving $C$ and exchanging $F_1$ and $F_2$ 
(which is an extension of a translation automorphism of $C$).  
Hence we also have $\widetilde{U}/{\aut}({\F}_2)=U/{\aut}({\F}_2)$. 
\end{proof}

Since $U$ is rational and ${\dim}\mathcal{M}_{18,0}=1$, we have 

\begin{proposition}\label{rational (18,0)}
The space $\mathcal{M}_{18,0}$ is rational. 
\end{proposition} 

The two points $p_1=F_1\cap C$, $p_2=F_2\cap C$ on the elliptic curve $C$ satisfy $2(p_1-p_2)\sim0$. 
This shows that $\mathcal{M}_{18,0}$ is naturally birational to the elliptic modular curve for $\Gamma_0(2)$ through the fixed curve map.

%%%%%%%%%%%%%%
% Section : g=0
%%%%%%%%%%%%%%

\section{The case $g=0$}\label{sec:g=0}

In this section we study the case $g=0$. 
The space $\mathcal{M}_{8,7}$ is unirational by the constructions in \cite{A-S} and \cite{A-S-T}, 
where a complete intersection model and an elliptic fibration model for the generic member are given respectively. 
Similarly, $\mathcal{M}_{10,6}$ is unirational by the quartic model given in \cite{A-S}. 
Here we shall present another triple cover construction for those two. 
The space $\mathcal{M}_{12,5}$ is birational to the moduli space of cubic surfaces (\cite{A-C-T}, \cite{D-G-K}), which is rational as is well-known. 

Below we (re)prove that ${\moduli}$ is unirational for $k\leq0$, and rational for $k\geq2$.  
As in \S \ref{sec:g=1}, even when ${\dim}{\moduli}\leq2$, 
we make a detour to present birational period maps.

\subsection{The unirationality of $\mathcal{M}_{8,7}$}\label{ssec:(8,7)}

We construct general members of $\mathcal{M}_{8,7}$ using certain triangles of 
anti-canonical curves on quadric del Pezzo surfaces. 
To begin with, 
let $U\subset|{\Oplane}(4)|\times({\proj}^2)^3$ be the locus of 
quadruplets $(C, p_1, p_2, p_3)$ such that 
$({\rm i})$ $C$ is a smooth quartic, 
$({\rm ii})$ $p_i\in C$, and 
$({\rm iii})$ if $L_i$ is the tangent line of $C$ at $p_i$, 
then $L_1$ (resp. $L_2$, $L_3$) passes through $p_2$ (resp. $p_3$, $p_1$). 
The space $U$ is rational of dimension $14$. 
Indeed, if we use the homogeneous coordinate of ${\proj}^2$ to normalize 
$p_1=[0,0,1]$, $p_2=[0,1,0]$, $p_3=[1,0,0]$ and express quartic forms as 
$\sum_{i,j,k}a_{ijk}X^iY^jZ^k$ with $i+j+k=4$, 
then the conditions $({\rm ii})$ and $({\rm iii})$ are given by 
\begin{equation*}
a_{400}=a_{301}=0, \quad a_{040}=a_{130}=0, \quad a_{004}=a_{013}=0. 
\end{equation*}
%Since the projection $U\to|{\Oplane}(4)|$ has finite fibers, it is also dominant. 

For $(C, p_1, p_2, p_3)\in U$, 
the double cover $\pi\colon Y\to{\proj}^2$ branched over $C$ is a quadric del Pezzo surface. 
The curves $D_i=\pi^{\ast}L_i$ are nodal $-K_Y$-curves such that 
\begin{equation*}
D_1\cap D_2={\rm Sing}(D_2), \quad D_2\cap D_3={\rm Sing}(D_3), \quad D_3\cap D_1={\rm Sing}(D_1). 
\end{equation*}
Then the curve $B=D_1+D_2+D_3$ has ordinary triple points at the nodes of $D_i$. 
We consider $\frac{1}{2}B$ as a mixed branch on $Y$ with all components shadow. 
The associated Eisenstein $K3$ surface has three isolated fixed points and no fixed curve. 
Thus we obtain a period map 
$\mathcal{P}\colon U/{\PGL}_3\dashrightarrow \mathcal{M}_{8,7}$.  

\begin{proposition}
The map $\mathcal{P}$ is dominant. 
\end{proposition}

\begin{proof}
Since ${\dim}(U/{\PGL}_3)={\dim}\mathcal{M}_{8,7}$, 
it suffices to show that $\mathcal{P}$ has countable fibers. 
The natural projection $g\colon\hat{Y}\to Y\to{\proj}^2$ is recovered from the degree $2$ line bundle $H=g^{\ast}{\Oplane}(1)$ 
as the associated projective morphism $\phi_H\colon\hat{Y}\to|H|^{\vee}$. 
Hence we have surjective maps onto the $\mathcal{P}$-fibers from subsets of ${\Pic}(\hat{Y})\simeq{\Z}^{11}$. 
\end{proof}

In this way, we obtain a proof of 

\begin{corollary}[cf. \cite{A-S}, \cite{A-S-T}]
The space $\mathcal{M}_{8,7}$ is unirational. 
\end{corollary}

\subsection{The unirationality of $\mathcal{M}_{10,6}$}

We consider a degeneration of our model for $\mathcal{M}_{8,5}$. 
Let $U\subset|{\Ospace}(3)|\times({\proj}^3)^2$ be the locus of triplets $(Y, p, q)$ such that 
(i) $Y$ is a smooth cubic surface containing $p$ and $q$, 
(ii) the $-K_Y$-curve $C_p=T_pY|_Y$ is irreducible and cuspidal, and 
(ii) the $-K_Y$-curve $C_q=T_qY|_Y$ is irreducible, nodal, and tangent to $C_p$ at $p$. 
Considering the mixed branches $C_q+\frac{1}{2}C_p$, we obtain Eisenstein $K3$ surfaces in $\mathcal{M}_{10,6}$. 
As before, one checks that the induced period map $U/{\PGL}_4\to\mathcal{M}_{10,6}$ is dominant. 
Since $U$ is rational, we have 

\begin{proposition}[cf. \cite{A-S}]\label{unirat (10,6)}
The space $\mathcal{M}_{10,6}$ is unirational. 
\end{proposition}

Using $C_q+C_p$ as $-2K_Y$-branches will give a canonical construction of general 2-elementary $K3$ surfaces of type $(15, 7, 1)$. %(\cite{Ma3}). 
Thus, via  $U/{\PGL}_4$ we have a natural birational map from an intermediate cover of ${\cove}_{10,6}\to\mathcal{M}_{10,6}$ 
to the orthogonal modular variety $\mathcal{M}_{15,7,1}$.

\subsection{The rationality of $\mathcal{M}_{14,4}$}

We consider curves on ${\F}_6$. 
Let $U\subset|L_{2,0}|$ be the locus of reducible curves $H_1+H_2$ such that 
$H_1, H_2$ are smooth members of $|L_{1,0}|$ transverse to each other. 
We associate the $-\frac{3}{2}K_{{\F}_6}$-curves $H_1+H_2+\Sigma$ to obtain 
a period map $\mathcal{P}\colon U/{\aut}({\F}_6)\dashrightarrow\mathcal{M}_{14,4}$. 

\begin{proposition}\label{period map (14,4)}
The map $\mathcal{P}$ is birational. 
\end{proposition}

\begin{proof}
We label independently the two curves $H_1, H_2$ and the six points $H_1\cap H_2$. 
This is realized by an $\frak{S}_2\times\frak{S}_6$-cover $\widetilde{U}\to U$. 
The two branches of $H_1+H_2+\Sigma$ at each of $H_1\cap H_2$ are distinguished by 
the given distinction of $H_1$ and $H_2$. 
Hence $\mathcal{P}$ lifts to a birational map 
$\widetilde{U}/{\aut}({\F}_6)\dashrightarrow{\cove}_{14,4}$ as before. 
Since ${\rm O}(A_L)\simeq {\rm GO}^-(4, 3)$ for the invariant lattice $L\simeq U\oplus E_6\oplus A_2^3$, 
the projection ${\cove}_{14,4}\to\mathcal{M}_{14,4}$ has degree $|{\rm GO}^-(4, 3)|/2=6!$ by \cite{Atlas}. 
On the other hand, the hyperelliptic involutions \eqref{eqn: HE invol} of $H_1+H_2$ exchange $H_1$ and $H_2$, 
so that the projection $\widetilde{U}/{\aut}({\F}_6)\to U/{\aut}({\F}_6)$ is an $\frak{S}_6$-covering. 
Therefore $\mathcal{P}$ has degree $1$. 
\end{proof}

\begin{proposition}\label{rational (14,4)}
The quotient $U/{\aut}({\F}_6)$ is rational. 
Therefore $\mathcal{M}_{14,4}$ is rational. 
\end{proposition}

\begin{proof}
We consider the ${\aut}({\F}_6)$-equivariant map 
$\varphi\colon U\dashrightarrow |L_{1,0}|$ defined in \eqref{average}. 
By Lemma \ref{linear system} $(2)$, we may apply the slice method for $\varphi$ to see that 
\begin{equation*}
U/{\aut}({\F}_6)\sim \varphi^{-1}(H)/G, 
\end{equation*}
where $H\in|L_{1,0}|$ is a smooth member and 
$G\simeq {\C}^{\times}\times{\PGL}_2$ is the stabilizer of $H$ in ${\aut}({\F}_6)$. 
Let $\iota_H$ be the involution of ${\F}_6$ which on each $\pi$-fiber $F$ 
fixes the two points $H|_F$, $\Sigma|_F$. 
Then $\varphi^{-1}(H)$ is an open set of the locus 
$\{ H'+\iota_H(H'), H'\in|L_{1,0}|\}$ in $|L_{2,0}|$. 
Thus $\varphi^{-1}(H)/G$ is birational to $(|L_{1,0}|/\iota_H)/G \sim |L_{1,0}|/G$.  
It is straightforward to see that 
the natural map $|L_{1,0}|\dashrightarrow|{\sheaf}_H(6)|$, $H'\mapsto H'|_H$, 
makes $\varphi^{-1}(H)/G$ birational to $|{\sheaf}_H(6)|/{\aut}(H)$. 
Then $|{\sheaf}_H(6)|/{\aut}(H)$ is birational to the moduli $\mathcal{M}_2$ of genus $2$ curves, 
which is rational by Igusa \cite{Ig}. 
\end{proof}

By the proof, we have a natural birational map $\mathcal{M}_{14,4}\dashrightarrow\mathcal{M}_2$. 
This might be related to the Janus example in \cite{H-W} Main Theorem (i).

\subsection{The rationality of $\mathcal{M}_{16,3}$}

We consider curves on ${\F}_4$. 
Let $U\subset|L_{2,0}|\times|L_{0,1}|$ be the locus of pairs $(H_1+H_2, F)$ 
such that $H_1, H_2$ are smooth members of $|L_{1,0}|$ transverse to each other, and $F$ is transverse to $H_1+H_2$. 
Considering the nodal $-\frac{3}{2}K_{{\F}_4}$-curves $H_1+H_2+F+\Sigma$, 
we obtain a period map 
$\mathcal{P}\colon U/{\aut}({\F}_4)\dashrightarrow\mathcal{M}_{16,3}$. 

\begin{proposition}\label{period map (16,3)}
The map $\mathcal{P}$ is birational. 
\end{proposition}

\begin{proof}
As in the proof of Proposition \ref{period map (14,4)}, 
we distinguish the two sections $H_1, H_2$ and the four points $H_1\cap H_2$ independently. 
This defines an $\frak{S}_2\times\frak{S}_4$-cover $\widetilde{U}\to U$. 
The rest datum for $H_1+H_2+F+\Sigma$ are then automatically labelled, 
and $\mathcal{P}$ will lift to a birational map 
$\widetilde{U}/{\aut}({\F}_4)\dashrightarrow{\cove}_{16,3}$. 
The projection $\widetilde{U}/{\aut}({\F}_4)\to U/{\aut}({\F}_4)$ is an $\frak{S}_4$-covering as before, 
while ${\cove}_{16,3}\to\mathcal{M}_{16,3}$ has degree $|{\rm O}(A_L)|/2$ 
for the invariant lattice $L=U\oplus E_8\oplus A_2^3$. 
It is straightforward to calculate that ${\rm O}(A_L)\simeq\frak{S}_3\ltimes({\Z}/2{\Z})^3$. 
\end{proof}

Since $U$ is unirational and ${\dim}\mathcal{M}_{16,3}=2$, we have 

\begin{proposition}\label{rational (16,3)}
The space $\mathcal{M}_{16,3}$ is rational. 
\end{proposition}

Arguing as in the proof of Proposition \ref{rational (14,4)}, 
one will see that $U/{\aut}({\F}_4)$ is naturally birational to the Kummer modular surface for 
${\SL}_2({\Z})$.

\subsection{The rationality of $\mathcal{M}_{18,2}$}

We consider curves on ${\F}_2$. 
Let $U\subset|L_{2,0}|\times|L_{0,2}|$ be the locus of pairs $(H_1+H_2, F_1+F_2)$ 
such that $H_1, H_2\in|L_{1,0}|$ are smooth and transverse to each other, 
and $F_1, F_2\in|L_{0,1}|$ are distinct and transverse to $H_1+H_2$. 
We associate the nodal $-\frac{3}{2}K_{{\F}_2}$-curves $H_1+H_2+F_1+F_2+\Sigma$ 
to obtain a period map 
$\mathcal{P}\colon U/{\aut}({\F}_2)\dashrightarrow\mathcal{M}_{18,2}$. 

\begin{proposition}\label{period map (18,2)}
The map $\mathcal{P}$ is birational. 
\end{proposition}

\begin{proof}
We distinguish independently the two sections $H_1, H_2$, the two fibers $F_1, F_2$, 
and the two points $H_1\cap H_2$. 
This is realized by an $(\frak{S}_2)^3$-cover $\widetilde{U}\to U$. 
As before, we see that these labelings induce a birational lift 
$\widetilde{U}/{\aut}({\F}_2)\dashrightarrow{\cove}_{18,2}$ of $\mathcal{P}$. 
Then ${\cove}_{18,2}$ is a double cover of $\mathcal{M}_{18,2}$ 
because we have ${\rm O}(A_L)\simeq({\Z}/2{\Z})^2$ 
for the invariant lattice $L=U(3)\oplus E_8^2$. 
On the other hand, the stabilizer in ${\aut}({\F}_2)$ of a general 
$(\sum_iH_i, \sum_iF_i)\in U$ is $({\Z}/2{\Z})^2$ 
generated by the hyperelliptic involution \eqref{eqn: HE invol} of $H_1+H_2$ 
and by an element exchanging the two points $H_1\cap H_2$ and the two fibers $F_1, F_2$ respectively. 
Thus $\widetilde{U}/{\aut}({\F}_2)\dashrightarrow U/{\aut}({\F}_2)$ 
is also a double covering. 
\end{proof}

Since $U$ is rational and ${\dim}\mathcal{M}_{18,2}=1$, we have 

\begin{proposition}\label{rational (18,2)}
The space $\mathcal{M}_{18,2}$ is rational. 
\end{proposition} 

Let $H=\varphi(H_1+H_2)$ be the section defined by \eqref{average}. 
As in the proof of Proposition \ref{rational (14,4)}, 
considering the configuration of $2+2$ points $H_1\cap H_2$, $F_1+F_2|_H$ on $H$ 
makes $U/{\aut}({\F}_2)$ birational to the elliptic modular curve for $\Gamma_0(2)$.

For completeness, we finish the article with a comment on $\mathcal{M}_{20,1}$, which consists of one point. 
Its unique member is obtained from the curve $\sum_{i=1}^{3}F_{i+}+\sum_{i=1}^{3}F_{i-}$ on ${\proj}^1\times{\proj}^1$, 
where $F_{i+}, F_{i-}$ are ruling fibers of bidegree $(1, 0), (0, 1)$ respectively.


\begin{thebibliography}{99}

\bibitem{A-N}Alexeev, V.; Nikulin, V. V.
\textit{Del Pezzo and K3 surfaces.} 
MSJ Memoirs, \textbf{15}. Math. Soc. Japan, 2006.

\bibitem{A-C-T}Allcock, D.; Carlson, J. A.; Toledo, D.
\textit{The complex hyperbolic geometry of the moduli space of cubic surfaces.} 
J. Algebraic Geom. \textbf{11} (2002), no. 4, 659--724. 

\bibitem{A-S}Artebani, M.; Sarti, A.
\textit{Non-symplectic automorphisms of order 3 on K3 surfaces.}
Math. Ann. \textbf{342} (2008), no. 4, 903--921. 

\bibitem{A-S-T}Artebani, M.; Sarti, A.; Taki, S. 
\textit{$K3$ surfaces with non-symplectic automorphisms of prime order.} 
Math. Z. \textbf{268} (2011), no. 1-2, 507--533. 

\bibitem{B-H-P-V}Barth, W. P.; Hulek, K.; Peters, C. A. M.; Van de Ven, A.
\textit{Compact complex surfaces.} 
Springer-Verlag,  2004. 

\bibitem{B-L}Birkenhake, C.; Lange, H. 
\textit{Complex abelian varieties.} 
(Second edition) Springer-Verlag, 2004.

\bibitem{Bo}Borel, A.
\textit{Some metric properties of arithmetic quotients of symmetric spaces and an extension theorem.} 
J. Diff. Geom. \textbf{6} (1972), 543--560. 

\bibitem{Atlas}Conway, J. H.; Curtis, R. T.; Norton, S. P.; Parker, R. A.; Wilson, R. A. 
\textit{ATLAS of finite groups.} 
Oxford University Press, 1985. 

\bibitem{Do}Dolgachev, I. V. 
\textit{Rationality of fields of invariants.}  
Algebraic geometry, Bowdoin, 1985, 3--16, 
Proc. Symp. Pure Math., \textbf{46}, Part 2, Amer. Math. Soc., Providence,  1987. 

\bibitem{D-G-K}Dolgachev, I.; van Geemen, B.; Kond\=o, S. 
\textit{A complex ball uniformization of the moduli space of cubic surfaces via periods of K3 surfaces.} 
J. Reine Angew. Math. \textbf{588} (2005), 99--148. 

\bibitem{D-K0}Dolgachev, I.; Kond\=o, S. 
\textit{Moduli of K3 surfaces and complex ball quotients.} 
Arithmetic and geometry around hypergeometric functions, 43--100, 
Progr. Math., \textbf{260}, Birkh\"auser, 2007. 

\bibitem{D-K}Dolgachev, I.; Kond\=o, S. 
\textit{The rationality of the moduli spaces of Coble surfaces and of nodal Enriques surfaces.} 
Izvestiya Math \textbf{77} (3), (2013) 509--524. 

%\bibitem{Fi}Fischer, E. 
%\textit{Die Isomorphie der Invariantenk\"orper der endlichen Abelschen Gruppen linearen Transformationen.} 
%Nachr. K\"onig. Ges. Wiss. G\"ottingen (1915), 77--80.
 
%\bibitem{Ge}Gerstein, L.~J. 
%\textit{Basic quadratic forms.}  
%Amer. Math. Soc., Providence, 2008.   
 
\bibitem{H-W}Hunt, B.; Weintraub, S. H. 
\textit{Janus-like algebraic varieties.} 
J. Diff. Geom. \textbf{39} (1994), no. 3, 509--557.

\bibitem{Ig}Igusa, J. 
\textit{Arithmetic variety of moduli for genus two.} 
Ann. of Math. (2) \textbf{72} (1960) 612--649. 

\bibitem{Ka1}Katsylo, P. I.
\textit{Rationality of the moduli spaces of hyperelliptic curves.} 
Izv. Akad. Nauk SSSR. \textbf{48} (1984), 705--710. 

\bibitem{Ka2}Katsylo, P. I. 
\textit{Rationality of the moduli variety of curves of genus $3$.} 
Comment. Math. Helv. \textbf{71} (1996), no. 4, 507--524.

\bibitem{Ki}Kitaoka, Y. 
\textit{Arithmetic of quadratic forms.}  
Cambridge University Press, 1993. 

\bibitem{Ko0}Kond\=o, S.
\textit{The rationality of the moduli space of Enriques surfaces.} 
Compositio Math. \textbf{91} (1994), 159--173. 

\bibitem{Ko}Kond\=o, S. 
\textit{The moduli space of curves of genus 4 and Deligne-Mostow's complex reflection groups.} 
Algebraic geometry 2000, Azumino, 383--400, 
Adv. Stud. Pure Math., 36, Math. Soc. Japan, 2002. 

\bibitem{Ma2} Ma, S. 
\textit{Rationality of the moduli spaces of $2$-elementary $K3$ surfaces.}  
arXiv:1110.5110, to appear in J. Alg. Geom.  

%\bibitem{Ma3} Ma, S. 
%\textit{Rationality of the moduli spaces of $2$-elementary $K3$ surfaces II.}  
%arXiv:1205.4783. 

\bibitem{M-M}Miranda, R.; Morrison, D. R.
\textit{The number of embeddings of integral quadratic forms. II.} 
Proc. Japan Acad. Ser. A Math. Sci. \textbf{62} (1986), no. 1, 29--32. 

\bibitem{Miy}Miyata, T. 
\textit{Invariants of certain groups. I.} 
Nagoya Math. J. \textbf{41} (1971), 69--73. 

\bibitem{Na}Naruki, I. 
\textit{Cross ratio variety as a moduli space of cubic surfaces.} 
Proc. London Math. Soc. \textbf{45} (1982), no. 1, 1--30. 

\bibitem{Ni0}Nikulin, V.V. 
\textit{Integral symmetric bilinear forms and some of their applications.} 
Math. USSR Izv., \textbf{14} (1980), 103--167.

%\bibitem{Ni1}Nikulin, V.V. 
%\textit{Finite automorphism groups of K\"ahlerian $K3$ surfaces.} 
%Trans. Moscow Math. Soc. \textbf{38} (1980), No 2, 71--135.

\bibitem{Ni}Nikulin, V.V.  
\textit{Factor groups of groups of automorphisms of hyperbolic forms with respect to subgroups generated by 2-reflections.} 
J. Soviet Math. \textbf{22} (1983), 1401--1476.

\bibitem{O-T}Ohashi, H.; Taki, S. 
\textit{$K3$ surfaces and log del Pezzo surfaces of index three.} 
Manuscripta Math. \textbf{139} (2012), no. 3-4, 443--471. 

\bibitem{Se}Segre, B.
\textit{The Non-singular Cubic Surfaces.} 
Oxford University Press, 1942. 

\bibitem{SB}Shepherd-Barron, N. I.
\textit{The rationality of certain spaces associated to trigonal curves.} 
Algebraic geometry, Bowdoin, 1985, 165--171, 
Proc. Symp. Pure Math., 46, Part 1, Amer. Math. Soc., Providence, RI, 1987.

\bibitem{Ta}Taki, S.
\textit{Classification of non-symplectic automorphisms of order 3 on $K3$ surfaces.} 
Math. Nachr. \textbf{284} (2011), no. 1, 124--135.

\end{thebibliography}
\end{document}